\documentclass[11pt]{article}
\usepackage{amsfonts}
\usepackage{amssymb}
\usepackage{amsthm}
\usepackage{amsmath}
\usepackage{graphicx}
\usepackage{empheq}
\usepackage{indentfirst}
\usepackage{cite} 
\usepackage{mathrsfs}
\usepackage{cases}
\usepackage{graphics}
\usepackage{xcolor}  
\usepackage{amsmath,bm}

\newtheorem{theorem}{\textbf{Theorem}}[section]
\newtheorem{lemma}{\textbf{Lemma}}[section]
\newtheorem{proposition}{\textbf{Proposition}}[section]
\newtheorem{corollary}{\textbf{Corollary}}[section]
\newtheorem{remark}{\textbf{Remark}}[section]
\newtheorem{definition}{\textbf{Definition}}[section]
\allowdisplaybreaks[4] 
\def\be{\begin{equation}}
\def\ee{\end{equation}}
\def\bea{\begin{eqnarray}}
\def\eea{\end{eqnarray}}
\def\bt{\begin{theorem}}
\def\et{\end{theorem}}
\def\bl{\begin{lemma}}
\def\el{\end{lemma}}
\def\br{\begin{remark}}
\def\er{\end{remark}}
\def\bp{\begin{proposition}}
\def\ep{\end{proposition}}
\def\bc{\begin{corollary}}
\def\ec{\end{corollary}}
\def\bd{\begin{definition}}
\def\ed{\end{definition}}

\begin{document}

\title{Global Well-posedness of a Navier--Stokes--Cahn--Hilliard System with Chemotaxis and Singular Potential in 2D}

\author{
Jingning He
\footnote{School of Mathematical Sciences, Fudan University, Handan Road 220, Shanghai 200433, China.
Email: \texttt{jingninghe2020@gmail.com}
}
\ and Hao Wu
\footnote{Corresponding author. School of Mathematical Sciences and Shanghai Key Laboratory for Contemporary Applied Mathematics, Fudan University, Handan Road 220, Shanghai 200433, China;
Key Laboratory of Mathematics for Nonlinear Sciences (Fudan University), Ministry of Education, Handan Road 220, Shanghai 200433, China.
Email: \texttt{haowufd@fudan.edu.cn}
}
}
\date{\today}
\maketitle


\begin{abstract}
\noindent We study a diffuse interface model that describes the dynamics of incompressible two-phase flows with chemotaxis effects. This model also takes into account some significant mechanisms such as active transport and nonlocal interactions of Oono's type. The system under investigation couples the Navier--Stokes equations for the fluid velocity, a convective Cahn--Hilliard equation with physically relevant singular potential for the phase-field variable and an advection-diffusion-reaction equation for the nutrient density. For the initial boundary value problem in a smooth bounded domain $\Omega\subset \mathbb{R}^2$, we first prove the existence and uniqueness of global strong solutions that are strictly separated from the pure states $\pm 1$ over time. Then we prove the continuous dependence with respect to the initial data and source term for the  strong solution in energy norms. Finally, we show the propagation of regularity for global weak solutions.
\end{abstract}

\section{Introduction}
\setcounter{equation}{0}
\noindent
In this paper, we consider the following Navier--Stokes--Cahn--Hilliard type system
\begin{subequations}
\begin{alignat}{3}
&\partial_t  \bm{ v}+\bm{ v} \cdot \nabla  \bm {v}-\textrm{div} (2  \eta(\varphi) D\bm{v} )+\nabla p=(\mu+\chi \sigma)\nabla \varphi,\label{f3.c} \\
&\textrm{div}\ \bm{v}=0,\label{f3.c1}\\
&\partial_t \varphi+\bm{v} \cdot \nabla \varphi=\Delta \mu-\alpha(\varphi-c_0),\label{f1.a} \\
&\mu=A\varPsi'(\varphi)-B\Delta \varphi-\chi \sigma,\label{f4.d} \\
&\partial_t \sigma+\bm{v} \cdot \nabla \sigma= \Delta (\sigma+\chi(1-\varphi))-\mathcal{C} h(\varphi) \sigma +S, \label{f2.b}
\end{alignat}
\end{subequations}
in $\Omega\times(0,T)$, where $\Omega \subset\mathbb{R}^2$ is a bounded domain with smooth boundary $\partial\Omega$ and $T>0$. The system is subject to the following boundary conditions
\begin{alignat}{3}
&\bm{v}=\mathbf{0},\quad {\partial}_{\bm{n}}\varphi={\partial}_{\bm{n}}\mu={\partial}_{\bm{n}}\sigma=0,\qquad\qquad &\textrm{on}& \   \partial\Omega\times(0,T),
\label{boundary}
\end{alignat}
and the initial conditions
\begin{alignat}{3}
&\bm{v}|_{t=0}=\bm{v}_{0}(x),\ \ \varphi|_{t=0}=\varphi_{0}(x), \ \ \sigma|_{t=0}=\sigma_{0}(x), \qquad &\textrm{in}&\ \Omega.
\label{ini0}
\end{alignat}
Here, $\bm{n}=\bm{n}(x)$ denotes the unit outward normal vector on $\partial\Omega$.

System \eqref{f3.c}--\eqref{f2.b} is a simplified version of the general thermodynamically consistent diffuse interface model that was derived in \cite{LW} for a two-phase incompressible fluid mixture with a chemical species subject to diffusion as well as some important transport mechanisms like advection and chemotaxis (see also \cite{S} and the references cited therein). This model can also be viewed as a Navier--Stokes analogue of the generalized Cahn--Hilliard--Darcy system in \cite{DFRSS,GLSS} for tumor growth modelling (see also \cite{EG19jde,EG19sima} for the associated  Cahn--Hilliard--Brinkman system). The fluid velocity $\bm{v}$ is taken as the volume-averaged velocity with $D\bm{v}=\frac{1}{2}(\nabla\bm{ v}+(\nabla\bm{ v}) ^ \mathrm{T})$ being the symmetrized velocity gradient, and the scalar function $p$ is the (modified) pressure. The order parameter $\varphi$ denotes the difference in volume fractions of the binary mixture such that the region $\{\varphi=1\}$ represents fluid 1 and $\{\varphi=-1\}$ represents fluid 2 (i.e., the values $\pm 1$ represent the pure concentrations). The variable $\sigma$ denotes the concentration of the chemical species, e.g., the nutrient.

 For the sake of simplicity, in this paper we assume that the density difference of the mixture as well as the mass transfer between the two components are negligible. The mobilities are assumed positive constants (set to be one), but we allow that the mixture may have unmatched viscosities. Assuming that $\eta_1$, $\eta_2>0$ are viscosities of the two homogeneous fluids, the mean viscosity can be modeled by the concentration dependent term $\eta=\eta(\varphi)$, for instance, a typical form could be the linear combination:
\be
\eta(\varphi)=\eta_1\frac{1+\varphi}{2}+\eta_2\frac{1-\varphi}{2}.
\label{vis}
\ee
 For possible nonlocal interactions (e.g. reversible chemical reaction) between the two components, here we adopt a simple form of Oono's type $-\alpha(\varphi-c_0)$, where $\alpha\geq 0$, $c_0\in(-1,1)$ are given constants (cf. \cite{GGM2017,Mi11}, for general source terms in biological applications, we refer to \cite{Fa15,Fa17}). The nutrient consumption is prescribed by the term $\mathcal{C}h(\varphi)\sigma$, where the non-negative constant $\mathcal{C}$ represents the consumption rate and the function $h$ is an interpolation with $h(-1)=0$ and $h(1)=1$, for instance, the simplest choice could be $h(\varphi) = \frac{1}{2}(1+\varphi)$ (cf. \cite{GLSS}). The function $S=S(x,t)$ stands for possible external source of the nutrient.

 In \eqref{f4.d}, $\mu$ denotes the chemical potential associated to $(\varphi, \sigma)$. The positive constants $A, B$ are related to the surface tension and the thickness of the interfacial layer (i.e., the diffuse interface). The constant coefficient $\chi$ is related to certain specific transport mechanisms such as chemotaxis and active transport in the context of tumor growth modelling (see e.g., \cite{GLSS, GL17e}). The nonlinear function $\varPsi'$ denotes the derivative of a potential $\varPsi$ with double-well structure, i.e., it has two minima and a local unstable maximum in between. A physically significant example is the logarithmic type \cite{CH}:
\be
\varPsi (r)=\frac{\theta}{2}[(1-r)\ln(1-r)+(1+r)\ln(1+r)]+\frac{\theta_{0}}{2}(1-r^2),\quad \forall\, r\in[-1,1],
\label{pot}
\ee
with $0<\theta<\theta_{0}$ (see, e.g., \cite{CH,CMZ}), which is sometimes referred to as Flory--Higgins potential. In the literature, the singular potential $\varPsi$ is often approximated by a fourth-order polynomial
\be
\varPsi(r)=\frac{1}{4}(1-r^2)^2,\quad r\in\mathbb{R}, \label{regular}
\ee
or some more general polynomial functions.

We mention two fundamental properties of problem \eqref{f3.c}--\eqref{ini0}. Integrating (formally) the equation \eqref{f1.a} over the domain $\Omega$, we obtain the following evolution equation for the spatial average of $\varphi$:
\begin{align}
\frac{d}{dt}(\overline{\varphi}-c_0)+\alpha(\overline{\varphi}-c_0)=0,
\quad \overline{\varphi}=\frac{1}{|\Omega|}\int_\Omega \varphi(x)dx,
\label{mph1}
\end{align}
from which we deduce
\begin{align}
\overline{\varphi}(t)-c_0=(\overline{\varphi_0}-c_0)e^{-\alpha t},\quad \forall\, t\geq 0.
\label{mph2}
\end{align}
Hence, if $\overline{\varphi_0}=c_0$, then the total mass (given by $\overline{\varphi}$) is conserved in time, otherwise, $\overline{\varphi}(t)$ converges exponentially fast to $c_0$ provided that $\alpha>0$ (the so-called off-critical case, cf. \cite{BGM,GGM2017,MT}). Next, as it has been shown in \cite{H,LW}, problem \eqref{f3.c}--\eqref{ini0} admits a basic energy law
\begin{align}
& \frac{d}{dt} \int_{\Omega} \Big[ \frac{1}{2}|\bm{v}|^2+A\varPsi(\varphi)+\frac{B}{2}|\nabla \varphi|^2+\frac{1}{2}|\sigma|^2+\chi\sigma(1-\varphi) \Big] dx \notag \\
& \qquad +\int_{\Omega} \Big[ 2\eta(\varphi)|D\bm{v}|^2 + |\nabla \mu|^2+|\nabla(\sigma+\chi(1-\varphi))|^2\Big] dx\nonumber\\
&\quad  =\int_\Omega \left[-\alpha(\varphi-c_0)\mu+(-\mathcal{C} h(\varphi) \sigma +S)(\sigma +\chi(1-\varphi))\right] dx,
\label{BEL}
\end{align}
which plays an important role in the study of its global well-posedness.

Problem \eqref{f3.c}--\eqref{ini0} was first analyzed in \cite{LW} under the choice of a regular potential including the prototype \eqref{regular}. The authors prove the existence of global weak solutions in two and three dimensions for prescribed mass transfer terms, the existence and uniqueness of global strong solutions in two dimensions. However, due to the coupling structure between $\varphi$ and $\sigma$ as well as the loss of maximal principle for the Cahn--Hilliard equation with regular potential, an artificial assumption on the coefficients $A$ and $\chi$ has to be imposed to prove the existence of global weak solutions (see also \cite{GL17e} for a similar situation in the fluid-free case with more general mass source terms). This extra assumption was removed in the recent work \cite{H}, where the author considered a singular potential like \eqref{pot} and proved the existence of global weak solutions in two and three dimensions. Besides, in the two dimensional case, the author established a continuous dependence result that yields the uniqueness of global weak solutions. The proof therein relies on the key property that the singular potential can guarantee the phase function $\varphi$ always stays in the physically relevant interval $[-1,1]$ during the time evolution, i.e., a uniform $L^\infty$ bound for $\varphi$ is available.

Our aim in this paper is to continue the study of \cite{H} and establish the global strong well-posedness of problem \eqref{f3.c}--\eqref{ini0} in the two dimensional setting, under the choice of a singular potential like \eqref{pot}.

First, we prove the existence and uniqueness of global strong solutions to problem \eqref{f3.c}--\eqref{ini0} subject to regular initial data $ \bm{v}_{0} \in \bm{H}^1(\Omega)$ with $\mathrm{div}\bm{v}_0=0$, $\varphi_{0} \in H^{2}(\Omega)$, $\sigma_{0}\in H^1(\Omega)$ with $\left\|\varphi_{0}\right\|_{L^{\infty}} \leq 1,\ \left|\overline{\varphi}_{0}\right|<1$ and
	$A\varPsi^{\prime}\left(\varphi_{0}\right)-B\Delta \varphi_{0} -\chi\sigma_0\in H^1(\Omega)$, which satisfy suitable boundary conditions (see Theorem \ref{2main} below).
The proof is based on the semi-Galerkin approximating scheme introduced in \cite{H}. In this framework, we perform a Galerkin approximation only for the Navier--Stokes equations \eqref{f3.c}, but solve the equations for $(\varphi, \sigma)$ independently. Again, with this approach we can take advantage of the fact $\varphi\in [-1,1]$ and drop the additional assumption on coefficients as in \cite{LW}.
We recall that the Navier--Stokes--Cahn--Hilliard--Oono system with a singular potential like \eqref{pot}, which can be viewed as a subsystem of our problem \eqref{f3.c}--\eqref{ini0} by neglecting the equation for $\sigma$, has been analyzed in \cite{MT} and the existence of a global weak solution was proven there via the standard Galerkin method (cf. \cite{BGM} for the case with regular potentials). Then our well-posedness result (Theorem \ref{2main}) naturally contains the existence and uniqueness of global strong solutions for the Navier--Stokes--Cahn--Hilliard--Oono system with singular potential in two dimensions.

Second, we prove the strict separation property of global strong solutions (see Theorem \ref{phase}). More precisely, if the initial datum $\varphi_0$ satisfies $\varphi_0\in[-1+\delta_0, 1-\delta_0]$ for some $\delta_0\in (0,1)$, then for every $T>0$, there exists a $\delta=\delta(T)>0$ such that $\|\varphi(t)\|_{L^{\infty}} \leq 1-\delta$ for all $t\in[0, T]$. Furthermore, with some additional assumptions, we can derive the uniform phase separation property on $[0,+\infty)$ such that there exists a $\delta\in(0,\delta_0]$, it holds $\|\varphi(t)\|_{L^{\infty}} \leq 1-\delta$ for all $t\geq 0$. The above property is usually called \textit{separation from pure states} in the literature. It is crucial for the study on Cahn--Hilliard equation with singular potential, since it enables people to overcome the singularity due to the potential $\varPsi$ and then obtain higher order regularity of solutions (see, for instance, \cite{A2007,CMZ,MZ04} for the Cahn--Hilliard equation, and see \cite{A2009,CG,G2019,GGW,GMT} for systems consisting of the Cahn--Hilliard equation coupled with different types of fluid interactions). For the Cahn--Hilliard--Oono equation with singular potential, the authors of \cite{GGM2017} showed that in dimension two, if the initial datum is not a pure state then the weak solution stays uniformly away from the pure states $\pm 1$ from a certain time on (and of course, the same holds for strong solutions). They presented an alternative proof that is different from those in \cite{A2007,MZ04} and can be extended to study more complicated systems \cite{CG,G2019,GGW,GMT}. Concerning the Navier--Stokes--Cahn--Hilliard--Oono system, as pointed out in \cite[Remarks 3.1, 3.2]{MT}, the simple linear term $\alpha(\varphi-c_0)$ in \eqref{f1.a} may cause the system to lose the conservation of mass when $\alpha>0$ (see \eqref{mph2}), which together with the coupling with the Navier--Stokes equations, yields essential difficulties for people to obtain (at least in two spatial dimensions) the separation property as in \cite{MZ04}. From the technical point of view, differentiation of the convective Cahn--Hilliard equation \eqref{f1.a} with respect to time leads to difficulties associated with the time derivative of the velocity $\partial_t\bm{v}$ and the extra singular term $\left(\varPsi^{\prime}\left(\varphi\right) \partial_t\overline{\varphi}, \partial_t (\varphi- \overline{\varphi})\right)$ in the corresponding energy estimates. We shall resolve this problem by extending the arguments in \cite{GGM2017,GMT}. Besides, it is worth noting that for our problem \eqref{f3.c}--\eqref{ini0}, some additional difficulties arise from the coupling with the nutrient equation for $\sigma$.
 In \cite{GGM2017,GMT}, an essential step to derive the separation property of $\varphi$ is to prove the singular derivative $\varPsi'(\varphi)\in L^{\infty}(\Omega\times (0, T))$ by using $\mu\in L^{\infty}(0, T ; H^2(\Omega))$ (see e.g., \cite[Section 5]{GGM2017} and \cite[Section 4]{GMT}, see also \cite{CG,GGW}). Following this strategy, for our problem \eqref{f3.c}--\eqref{ini0}, we need to show $\mu+\chi\sigma\in L^{\infty}(0, T ; H^2(\Omega))$. However, the regularity $\sigma\in L^{\infty}(0, T; H^2(\Omega))$ is not available from the nutrient equation \eqref{f2.b} if we only assume $\sigma_0\in H^1(\Omega)$. To overcome this difficulty, we employ a different approach, by exploiting the estimate for an elliptic problem with singular nonlinear term (see Lemma \ref{se} below). Indeed, only by making use of the fact $\mu+\chi\sigma\in L^{\infty}(0, T; H^1(\Omega))$, we are able to deduce that $\varPsi'(\varphi)\in L^{\infty}(0, T; W^{1,q}(\Omega))$ for any $q\in[2,+\infty)$, which implies $\varPsi'(\varphi)\in L^{\infty}(0, T ; L^{\infty}(\Omega))$ thanks to the Sobolev embedding theorem in two dimensions. This together with the assumption that $\varPsi'$ diverges at $\pm 1$ and the known regularity $\varphi\in L^\infty(0,T; C(\overline{\Omega}))$ yields the strict separation of $\varphi$ from the pure states $\pm 1$.

 Finally, the separation property of strong solutions paves the road for two further results. We note that
 with the separation property, the Cahn--Hilliard equation with singular potential can be regarded as a parabolic equation with a globally Lipschitz smooth potential. This allows us to derive the continuous dependence on the initial data for strong solutions in higher-order norms (see Corollary \ref{2main1} and cf. \cite{H} for the continuous dependence estimate of global weak solutions). On the other hand, thanks to the parabolic nature of the system, we are able to show the propagation of regularity for any global weak solution to problem \eqref{f3.c}--\eqref{ini0} and in particular, the instantaneous separation property from pure states $\pm 1$ for positive time (see Corollary \ref{regw}).

 Neglecting the interaction with nutrient and setting $\alpha=0$, our system \eqref{f3.c}--\eqref{f4.d} reduces to the well-known ``Model H" for the dynamics of incompressible viscous two-phase flows \cite{HH, Gur}, which has been extensively analyzed in the literature under various assumptions on fluid viscosity, mobility and potential function. For instance, we refer to \cite{B,GG2010,LS,ZWH} for the Navier--Stokes--Cahn--Hilliard system with regular potentials, and to \cite{A2009,B,GMT} for the case with singular (e.g., logarithmic) potentials. The system for two-phase fluids with nonlocal effect due to Oono's interaction was analyzed in \cite{BGM,MT}. On the other hand, when the fluid interaction in system \eqref{f3.c}--\eqref{f4.d} is neglected, we refer to \cite{GL17e} for the well-posedness of a Cahn--Hilliard type system with chemotaxis and active transport (see \cite{GL17} for the case with Dirichlet boundary conditions) and to \cite{MRS} for long-time behavior without transport mechanisms. At last, we would like to mention some recent contributions
\cite{DGG,DG,FGG16,FGGS,FG12,FGR15,GGG17} on the nonlocal Cahn--Hilliard equation and its coupled systems for two-phase fluids. The nonlocal variant of  system \eqref{f3.c}--\eqref{f4.d} would be an interesting subject of future study.

 The remaining part of this paper is organized as follows. In Section \ref{pm}, we introduce the functional settings and state the main results.  Section \ref{ws} is devoted to the existence and uniqueness of global strong solutions in two dimensions. In Section \ref{phs}, we prove the strict  separation property for the strong solution $\varphi$ on an arbitrary interval $[0, T]$ as well as a uniform phase separation property on $[0,+\infty)$ under some additional assumptions on the coefficients of the system. In Section \ref{cd}, we address the continuous dependence with respect to the initial data of the strong solutions. In the last Section \ref{phsw}, we prove the instantaneous regularity of weak solutions.

\section{Main Results}\label{pm}
\setcounter{equation}{0}
\subsection{Preliminaries}
We assume that $\Omega \subset\mathbb{R}^2$ is a bounded domain with smooth boundary $\partial\Omega$ and $T>0$ is a fixed final time. For the standard Lebesgue and Sobolev spaces, we use the notations $L^{p} := L^{p}(\Omega)$ and $W^{k,p} := W^{k,p}(\Omega)$ for any $p \in [1,+\infty]$ and $k > 0$, equipped with the norms $\|\cdot\|_{L^{p}}$ and $\|\cdot\|_{W^{k,p}}$.  When $p = 2$, we use $H^{k} := W^{k,2}$ and the norm $\|\cdot\|_{H^{k}}$. The norm and inner product on $L^{2}(\Omega)$ are simply denoted by $\|\cdot\|$ and $(\cdot,\cdot)$, respectively.
The dual space of a Banach space $X$ is denoted by $X'$, and the duality pairing between $X$ and $X'$ will be denoted by
$\langle \cdot,\cdot\rangle_{X',X}$. Given an interval $J$ of $\mathbb{R}^+$, we introduce
the function space $L^p(J;X)$ with $p\in [1,+\infty]$, which
consists of Bochner measurable $p$-integrable
functions with values in the Banach space $X$.
The boldface letter $\bm{X}$ denotes the vectorial space $X^d$ endowed with the product structure.

For every $f\in H^1(\Omega)'$, we denote by $\overline{f}$ its generalized mean value over $\Omega$ such that
$\overline{f}=|\Omega|^{-1}\langle f,1\rangle_{(H^1)',\,H^1}$. If $f\in L^1(\Omega)$, then its mean is simply given by $\overline{f}=|\Omega|^{-1}\int_\Omega f \,dx$.
  We will use the Poincar\'{e}--Wirtinger inequality
\begin{equation}
\label{poincare}
\|f-\overline{f}\|\leq C_P\|\nabla f\|,\quad \forall\,
f\in H^1(\Omega),
\end{equation}
where $C_P$ is a constant depending only on $\Omega$. Then we
see that $f\to (\|\nabla f\|^2+|\overline{f}|^2)^\frac12$ is an equivalent norm on $H^1(\Omega)$. We introduce the space $L^2_{0}(\Omega):=\{f\in L^2(\Omega):\overline{f} =0\}$ for $L^2$ functions with zero mean and set
$H^2_{N}(\Omega):=\{f\in H^2(\Omega):\,\partial_{\bm{n}}f=0 \ \textrm{on}\  \partial \Omega\}$ for $H^2$ functions subject to homogeneous Neumann boundary condition. The Ladyzhenskaya and Agmon's inequalities in two dimensions will be used in the subsequent estimates
\begin{align*}
&\|f\|_{L^4}\leq C\|f\|_{H^1}^\frac12\|f\|^\frac12,\qquad \forall\,f\in H^1(\Omega),\\
&\|f\|_{L^\infty}\leq  C\|f\|_{H^2}^\frac12\|f\|^\frac12,\qquad \forall\,f\in H^2(\Omega).
\end{align*}

Next, we recall some function spaces for the Navier--Stokes equations (see e.g., \cite{S}). Let $\bm{C}^\infty_{0,\mathrm{div}}(\Omega)$ be the space of divergence free vector fields in $(C^\infty_0(\Omega))^2$. We define $\bm{L}^2_{0,\mathrm{div}}(\Omega)$ and $\bm{H}^1_{0,\mathrm{div}}(\Omega)$ as the closure of $\bm{C}^\infty_{0,\mathrm{div}}(\Omega)$ with respect to the $\bm{L}^2$ and $\bm{H}^1_0$ norms, respectively. The space $\bm{H}^1_{0,\mathrm{div}}(\Omega)$ is equipped with the scalar product
 $(\bm{u},\bm{v})_{\bm{H}^1_{0,\mathrm{div}}}:=(\nabla \bm{u},\nabla \bm{v})$ for all $\bm{u},\, \bm{v}  \in {\bm{H}^1_{0,\mathrm{div}}(\Omega)}$ and the norm $\|\bm{u}\|_{\bm{H}^1_{0,\mathrm{div}}}=\|\nabla \bm{u}\|$. The classical Korn's inequality
 entails that $\|\nabla \bm{u}\| \leq \sqrt{2} \|D\bm{u}\|\leq \sqrt{2}\|\nabla \bm{u}\|$ for all $\bm{u}\in \bm{H}^1_{0,\mathrm{div}}(\Omega)$.
 It is well known that $\bm{L}^2(\Omega)$ can be decomposed into $\bm{L}^2_{0,\mathrm{div}}(\Omega)\oplus\bm{G}(\Omega)$, where $\bm{G}(\Omega):=\{\bm{f}\in\bm{L}^2(\Omega): \exists\, g\in H^1(\Omega),\ \bm{f}=\nabla g\}$. Then for any function $\bm{f} \in \bm{L}^2(\Omega)$, there holds the Helmholtz--Weyl decomposition:
$\bm{f}=\bm{f}_{0}+\nabla g$ where  $\bm{f}_{0} \in \bm{L}^2_{0,\mathrm{div}}(\Omega)$ and $\nabla g \in \bm{G}(\Omega)$. Consequently, one can define the Helmholtz--Leray projection onto the space of divergence-free functions $\bm{P}:\bm{L}^2(\Omega)\to \bm{L}^2_{0,\mathrm{div}}(\Omega)$ such that $\bm{P}(\bm{f})=\bm{f}_{0}$. Recall now the Stokes operator $\bm{S}: \bm{H}^1_{0,\mathrm{div}}(\Omega)\cap\bm{H}^2(\Omega)\to\bm{L}^2_{0,\mathrm{div}}(\Omega)$ such that
\be
(\bm{S}\bm{u},\bm{\zeta})=(\nabla \bm{u},\nabla\bm{\zeta}),\quad  \forall\, \bm{\zeta} \in \bm{H}^1_{0,\mathrm{div}}(\Omega),\nonumber
 \ee
 with domain $D(\bm{S})= \bm{H}^1_{0,\mathrm{div}}(\Omega)\cap\bm{H}^2(\Omega)$ (see e.g., \cite[Chapter III]{S}). Besides, there exists a constant $C>0$ such that $\|\bm{u}\|_{\bm{H}^2}\leq \|\bm{S} \bm{u}\|$ for any $\bm{u}\in \bm{H}^1_{0,\mathrm{div}}(\Omega)\cap\bm{H}^2(\Omega)$. The operator $\bm{S}$ is a canonical isomorphism from $\bm{H}^1_{0,\mathrm{div}}(\Omega)$ to $\bm{H}^1_{0,\mathrm{div}}(\Omega)'$.
 Denote its inverse map by $\bm{S}^{-1}:\bm{H}^1_{0,\mathrm{div}}(\Omega)'\to\bm{H}^1_{0,\mathrm{div}}(\Omega)$. For any $\bm{f}\in \bm{H}^1_{0,\mathrm{div}}(\Omega)'$, there is a unique $\bm{u}=\bm{S}^{-1}\bm{f}\in\bm{H}^1_{0,\mathrm{div}}(\Omega)$ such that
\be
(\nabla\bm{S}^{-1}\bm{f},\nabla \bm{\zeta})=\langle\bm{f},\bm{\zeta}\rangle_{(\bm{H}^1_{0,\mathrm{div}})',\,\bm{H}^1_{0,\mathrm{div}}},\quad \forall\, \bm{\zeta} \in \bm{H}^1_{0,\mathrm{div}}(\Omega),\nonumber
\ee
and $\|\nabla\bm{S}^{-1}\bm{f}\|=\langle\bm{f},\bm{S}^{-1}\bm{f} \rangle_{(\bm{H}^1_{0,\mathrm{div}})',\,\bm{H}^1_{0,\mathrm{div}}}^{\frac{1}{2}}$ is an equivalent norm on $\bm{H}^1_{0,\mathrm{div}}(\Omega)'$.
Finally, we recall the following regularity result (see e.g.,  \cite[Chapter III, Theorem 2.2.1]{S} and \cite[Appendix B]{GMT}):
\bl \label{stokes}
 \rm Let $\Omega \subset\mathbb{R}^2$ be a bounded smooth domain. For any $\bm{f} \in \bm{L}^2_{0,\mathrm{div}}(\Omega)$,
there exists a unique pair $\bm{u}\in \bm{H}^1_{0,\mathrm{div}}(\Omega)\cap\bm{H}^2(\Omega)$ and $p\in H^1(\Omega)\cap L_0^2(\Omega)$ such that $-\Delta \bm{u}+\nabla p=\bm{f}$ a.e. in $\Omega$, that is, $\bm{u}=\bm{S}^{-1}\bm{f}$. Moreover,
\begin{align*}
&\|\bm{u}\|_{\bm{H}^2}+\|\nabla p\|\le C\|\bm{f}\|,\quad
\|p\|\le C \|\bm{f}\|^\frac12\|\nabla \bm{S}^{-1}\bm{f}\|^\frac12,
 \end{align*}
where $C$ is a positive constant that may depend on $\Omega$ but is independent of $\bm{f}$.
 \el

Throughout this paper, the symbols $C$, $C_i$, $i\in \mathbb{N}$, denote generic positive constants that may depend on  coefficients of the system, norms of the initial data, $\Omega$ and $T$. Their values may change from line to line and specific dependencies will be pointed out when necessary.

\subsection{Assumptions}
\noindent
We introduce some basic hypotheses for problem  \eqref{f3.c}--\eqref{ini0} (cf. \cite{H,GMT,GGW}).
\begin{enumerate}
	\item[(H1)]\label{seta} The fluid viscosity $\eta$ satisfies  $\eta \in C^{2}(\mathbb{R})$  and
	$\eta_{*} \leq \eta(r)\leq \eta^*$ for $r \in \mathbb{R}$, where $\eta_{*}<\eta^*$ are positive constants.
	\item[(H2)]\label{item:as} The singular potential $\varPsi$ belongs to the class of functions $C[-1,1]\cap C^{3}(-1,1)$ and can be written into the following form
	\begin{equation}
	\varPsi(r)=\varPsi_{0}(r)-\frac{\theta_{0}}{2}r^2,\nonumber
	\end{equation}
	such that
	\begin{equation}
	\lim_{r\to \pm 1} \varPsi_{0}'(r)=\pm \infty ,\quad \text{and}\ \  \varPsi_{0}''(r)\ge \theta,\quad \forall\, r\in (-1,1),\nonumber
	\end{equation}
	where $\theta$ is a strictly positive constant and $\theta_0\in \mathbb{R}$. We make the extension $\varPsi_{0}(r)=+\infty$ for any $r\notin[-1,1]$.
	In addition, there exists $\epsilon_0\in(0,1)$ such that $\varPsi_{0}''$ is nondecreasing in $[1-\epsilon_0,1)$ and nonincreasing in $(-1,-1+\epsilon_0]$.
	\item[(H3)] The convex potential $\varPsi_0$ satisfies the following growth assumption on its derivatives
	\be \varPsi_0^{\prime \prime}(z) \leq C \mathrm{e}^{C\left|\varPsi_0^{\prime}(z)\right|}, \quad \forall\, z \in(-1,1). \label{de2}\ee
	 for some positive constant  $C$.
	\item[(H4)] The functions $h$, $S$ satisfy $h\in C^1(\mathbb{R})\cap L^\infty(\mathbb{R})$ and $S\in L^2(0,T; L^2(\Omega))$.
	\item[(H5)]\label{sco} The coefficients $A,\ B,\ \mathcal{C},\ \chi,\ \alpha,\ c_0$ are  prescribed  constants and fulfill
	\be
	A>0,\ \ B>0,\ \ \mathcal{C}\in\mathbb{R},\ \ \chi \in \mathbb{R},\ \ \alpha\geq 0,\ \ c_0\in(-1,1). \nonumber
	\ee
\end{enumerate}
\begin{remark}
	(1) As indicated in \cite[Remark 2.1]{GMT}, one can easily extend the linear viscosity function \eqref{vis} to $\mathbb{R}$ in such a way to comply $\mathrm{(H1)}$. Indeed, since the singular potential $\varPsi$ guarantees that the solution $\varphi\in [-1,1]$, the value of $\eta$ outside of $[-1,1]$ is not important and can be chosen in a good manner as in $\mathrm{(H1)}$.
	(2) The logarithmic potential \eqref{pot} satisfies the assumptions $\mathrm{(H2)}$ and $\mathrm{(H3)}$.
(3) We note that here we have removed the assumption about the convexity of the second order derivative $\varPsi_0''$ that was previously required in \cite{GGW,GMT}.
\end{remark}

\subsection{Statement of main results}
We are now in a position to present the main results of this paper.
\begin{theorem}[Strong well-posedness] \label{2main}
	Let $\Omega \subset \mathbb{R}^{2}$ be a bounded smooth domain and $T>0$. Suppose that the hypotheses (H1)--(H5) are satisfied. For any $ \bm{v}_{0} \in \bm{H}^1_{\mathrm{0,div}}(\Omega)$, $\varphi_{0} \in H^{2}_N(\Omega)$, $\sigma_{0}\in H^1(\Omega)$ with $\left\|\varphi_{0}\right\|_{L^{\infty}} \leq 1,\ \left|\overline{\varphi}_{0}\right|<1$ and
	$\mu_{0}=A\varPsi^{\prime}\left(\varphi_{0}\right) -B\Delta \varphi_{0} -\chi\sigma_0\in H^1(\Omega)$,  there exists a unique global strong
	solution $(\bm{v}, p, \varphi, \mu, \sigma)$ to problem \eqref{f3.c}--\eqref{ini0} on $[0, T]$ such that
	\begin{align*}
	&\bm{v} \in C\left([0, T] ; \bm{H}^1_{\mathrm{0,div}}(\Omega)\right) \cap L^{2}\left(0, T ; \bm{H}^1_{\mathrm{0,div}}(\Omega)\cap\bm{H}^2(\Omega)\right) \cap H^{1}\left(0, T ;\bm{L}^2_{0,\mathrm{div}}(\Omega)\right), \\
	&p\in L^2(0,T;H^1(\Omega)),\\
	&\varphi \in C_w\left([0, T] ; H^3(\Omega)\right) \cap L^2(0,T;H^4(\Omega))\cap H^{1}(0, T ; H^1(\Omega)),\\
	& \mu\in C([0, T] ; H^1(\Omega)) \cap L^{2}\left(0, T ; H^{3}(\Omega)\right) \cap H^{1}\left(0, T ; H^1(\Omega)^{\prime}\right),\\
	&\sigma   \in C([0,T] ;H^1(\Omega))\cap L^{2}(0,T;H^2(\Omega)) \cap H^1(0,T;L^2(\Omega)),\\
	&\varphi\in L^{\infty}(\Omega\times (0,T))\ \textrm{and}\ \ |\varphi(x,t)|<1\ \ \textrm{a.e.\ in}\ \Omega\times(0,T)\notag
	\end{align*}
	The strong solution satisfies the equations \eqref{f3.c}--\eqref{f2.b} a.e. in $\Omega \times(0, T)$, the boundary conditions \eqref{boundary} a.e. on $\partial\Omega\times(0,T)$ as well as the initial conditions \eqref{ini0} a.e. in $\Omega$.
	\end{theorem}
\begin{remark}
Like in \cite{CG,GGW,GMT}, the admissible initial datum $\varphi_0$ has finite energy, namely, one can infer from $\|\varphi_0\|_{L^\infty}\leq 1$ that $\varPsi(\varphi_0)\in L^1(\Omega)$. On the other hand, the assumption on $\overline{\varphi_0}$ prevents the admissibility of the pure phases $\pm1$.
\end{remark}

Our second result concerns the strict separation property of the phase function $\varphi$, i.e., it always stays within a closed subset of $(-1, 1)$ over time.
 \begin{theorem}[Separation from pure states] \label{phase}
  Let the assumptions in Theorem \ref{2main} be satisfied.

  (1) For the global strong solution to problem \eqref{f3.c}--\eqref{ini0} obtained in Theorem \ref{2main}, there exists a constant $\delta_T\in (0,1)$ such that the phase function $\varphi$ satisfies
\be
\|\varphi(t)\|_{L^{\infty}} \leq 1-\delta_T,\quad \forall\, t \in[0,T],\label{sep1}
\ee
where $\delta_T$ may depend on norms of the initial data, coefficients of the system, $\Omega$ and $T$, but is independent of the specific initial datum.

(2) If in addition, we assume that
\be
\mathcal{C}= 0,\qquad \|S\|_{L^1(0,+\infty, L^1(\Omega))}+ \sup_{t\in [0,+\infty)}\int_t^{t+1}\|S(s)\|^2\,ds\leq S_0,
\label{ass1}
\ee
where $S_0>0$ is an arbitrary but fixed constant, then there exists a constant  $\widehat{\delta}\in (0,1)$ such that
\be
\|\varphi(t)\|_{L^{\infty}} \leq 1-\widehat{\delta},\quad \forall\, t \in[0,+\infty),\label{2sep1}
\ee
in particular,  $\widehat{\delta}$ is independent of time.
\end{theorem}

Thanks to the separation property obtained in Theorem \ref{phase}, we can derive a refined continuous dependence estimate for the strong solutions with respect to initial data in the energy space.
 \begin{corollary}[Continuous dependence]
 \label{2main1}
 	Let the assumptions in Theorem \ref{2main} be satisfied. Consider two sets of initial data satisfying  $\bm{v}_{0i}\in\bm{H}^1_{0,\mathrm{div}}(\Omega)$, $\varphi_{0i}\in H^2_N(\Omega)$, $\sigma_{0i}\in H^1(\Omega)$ with $\left \|  \varphi_{0i}\right \|_{L^{\infty}} \le 1$, $|\overline{\varphi}_{0i}|<1$, and 	$\mu_{0i}=A\varPsi^{\prime}\left(\varphi_{0i}\right) -B\Delta \varphi_{0i} -\chi\sigma_{0i}\in H^1(\Omega)$, and two source terms $S_i$, $i=1,\, 2$.
 	Denote the corresponding global strong solutions to problem \eqref{f3.c}--\eqref{ini0} by $(\bm{v}_{i},\varphi_{i},\mu_i,\sigma_{i})$, $i=1,\, 2$. Then we have the following estimate:
 	 \begin{align}
&\|\boldsymbol{v}_1(t)-\boldsymbol{v}_2(t)\|^{2}
+\|\varphi_2(t)-\varphi_2(t)\|^{2}_{H^1}
+\|\sigma_1(t)-\sigma_2(t)\|^{2}\nonumber \\
&\quad \le C_T\Big(\|\boldsymbol{v}_{01}-\boldsymbol{v}_{02}\|^{2}
+\|\varphi_{01}-\varphi_{02}\|^{2}_{H^1}
+\|\sigma_{01}-\sigma_{02}\|^{2}+\int_0^t\|S_1(s)-S_2(s)\|^2\,ds \Big)\label{conti1}
 	\end{align}
 	for all $t\in [0,T]$, where $C_T$ is a positive constant depending on norms of the initial data, coefficients of the system, $\Omega$ and $T$.
 	\end{corollary}

Finally, the parabolic nature of problem \eqref{f3.c}--\eqref{ini0} enables us to obtain some further regularities of its global weak solutions in two dimensions. To this end, we recall the following  definition (see \cite{H}):
\bd \label{maind}
Let $\Omega \subset \mathbb{R}^{2}$ and $T \in (0,+\infty)$. Suppose that the initial data satisfy $\bm {v}_{0} \in \bm {L}^2_{0,\mathrm{div}}(\Omega)$, $\varphi_{0}\in H^1(\Omega)$, $\sigma_{0}\in L^2(\Omega)$ with $\|  \varphi_{0} \|_{L^{\infty}} \le 1$ and
$|\overline{\varphi}_{0}|<1$. A quadruple $(\bm{v},\varphi,\mu,\sigma)$ satisfying the following properties
\begin{align}
	&\bm{v} \in L^{\infty}(0,T;\bm{L}^2_{0,\mathrm{div}}(\Omega)) \cap L^{2}(0,T;\bm{H}^1_{0,\mathrm{div}}(\Omega))\cap  H^{1}(0,T;\bm{H}^1_{0,\mathrm{div}}(\Omega)')\notag,\\
	&\varphi \in L^{\infty}(0,T;H^1(\Omega))\cap L^{4}(0,T;H^2_{N}(\Omega))\cap L^2(0,T;W^{2,q}(\Omega)) \cap H^{1}(0,T;H^1(\Omega)'),\notag \\
	&\mu \in   L^{2}(0,T;H^1(\Omega)),\notag \\
	&\sigma  \in L^{\infty}(0,T;L^2(\Omega))\cap L^{2}(0,T;H^1(\Omega)) \cap H^{1}(0,T;H^1(\Omega)'),\notag\\
	&\varphi\in L^{\infty}(\Omega\times (0,T))\ \textrm{and}\ \ |\varphi(x,t)|<1\ \ \textrm{a.e.\ in}\ \Omega\times(0,T),\notag
\end{align}
where $q\in [2, +\infty)$, is called a weak solution to problem \eqref{f3.c}--\eqref{ini0} on $[0,T]$, if
\begin{subequations}
	\begin{alignat}{3}
		&\left \langle\partial_t  \bm{ v},\bm{\zeta}\right \rangle_{(\bm{H}^1_{0,\mathrm{div}})',\,\bm{H}^1_{0,\mathrm{div}}}+(\bm{ v} \cdot \nabla  \bm {v},\bm{ \zeta})+(  2\eta(\varphi) D\bm{v},D\bm{ \zeta}) =((\mu+\chi \sigma)\nabla \varphi,\bm {\zeta}),\notag \\
		&\left \langle \partial_t \varphi,\xi\right \rangle_{(H^1)',\,H^1}+(\bm{v} \cdot \nabla \varphi,\xi)=- (\nabla \mu,\nabla \xi)-\alpha(\varphi-c_0,\xi),\notag \\
		&\left \langle\partial_t \sigma,\xi\right \rangle_{(H^1)',\,H^1}+(\bm{v} \cdot \nabla \sigma,\xi) + (\nabla \sigma,\nabla \xi) = \chi ( \nabla \varphi,\nabla \xi)-(\mathcal{C} h(\varphi) \sigma,\xi) + (S,\xi), \notag
	\end{alignat}
\end{subequations}
a.e. in $(0,T)$ for all $\bm {\zeta} \in \bm{H}^1_{0,\mathrm{div}}$ and $\xi \in H^1(\Omega)$, where $ \mu=A\varPsi'(\varphi)-B\Delta \varphi-\chi \sigma$ a.e. in $\Omega\times (0,T)$. Moreover, the initial conditions in \eqref{ini0} are fulfilled.
\ed
When the spatial dimension is two, the existence and uniqueness of global weak solutions have been proven in \cite[Theorems 2.1, 2.2]{H}. Based on Theorems \ref{2main}, \ref{phase}, we can show the instantaneous regularity of global weak solutions as well as the separation property for $\varphi$ on the time interval $[\tau,T]$ for any $\tau \in(0,T)$.

\begin{corollary}[Instantaneous regularity of weak solutions]
\label{regw}
Let $(\bm{v},\varphi,\mu,\sigma)$ be the global weak solution to problem \eqref{f3.c}--\eqref{ini0} on $[0,T]$ as stated in Definition \ref{maind}. Suppose that the hypotheses (H1)--(H5) are satisfied. Then for any $\tau\in(0,T)$, the weak solution satisfies those regularity properties of strong solutions as stated in Theorem \ref{2main} on the time interval $[\tau, T]$. Besides, there exists a constant $\widetilde{\delta}\in (0,1)$, such that
\be
\|\varphi(t)\|_{L^{\infty}} \leq 1-\widetilde{\delta},\quad \forall\, t \in[\tau,T],\label{sep1w}
\ee
where $\widetilde{\delta}$ may depend on $\|\bm{v}_0\|$, $\|\varphi_0\|_{H^1}$, $\|\sigma_0\|$, $\int_\Omega \varPsi(\varphi_0)dx$, coefficients of the system, $\Omega$, $T$ and $\tau$, but is independent of the specific initial datum.
\end{corollary}
Before ending this section, we give two more remarks on the structure of the system \eqref{f3.c}--\eqref{f2.b} and some possible extensions of the above results.

\begin{remark}
As mentioned in the Introduction, our system \eqref{f3.c}--\eqref{f2.b} is a simplified version of the thermodynamically consistent model derived in \cite{LW} for a two-component fluid mixture, which allows for mass transfer between the two components, in the presence of a chemical species that is subject to diffusion, convection and chemotaxis. In the general framework therein, the chemical potential and the flux for the nutrient $\sigma$ are given by
$$
\mu=A\varPsi'(\varphi)-B\Delta\varphi+N_\varphi,\qquad
\mathbf{q}_\sigma = -n(\varphi)\nabla N_\sigma,
$$
 where the functions $N_\varphi$ and $N_\sigma$
are the partial derivatives of the chemical free energy density $N(\varphi, \sigma)$ with respect to $\varphi$ and $\sigma$, respectively. The function $n(\varphi)>0$ denotes the mobility for the nutrient. By taking the following typical form of the chemical free energy density $N$ (see e.g., \cite{GLSS,LW}):
\begin{equation}
	N(\varphi, \sigma)=\frac{\chi_\sigma}{2}|\sigma|^{2}+\chi_\varphi \sigma(1-\varphi), \quad \text{for some}\ \chi_\varphi\in \mathbb{R},\ \chi_\sigma>0,
\label{NN}
\end{equation}
we have
	$$
	N_{\varphi}=-\chi_\varphi \sigma, \quad N_{\sigma}=\chi_\sigma\sigma+\chi_\varphi(1-\varphi),
	$$
which is exactly the case in our system \eqref{f3.c}--\eqref{f2.b} by setting $\chi_\sigma=1$ and $\chi_\varphi=\chi$ (together with $n(\varphi)=1$). On the other hand, in view of \eqref{NN} and using the argument in \cite[Section 1]{GLSS}, we can, for example, make the following choice for the mobility $n(\varphi)$ and the diffusion coefficient $\chi_\sigma$ such that
$$
n(\varphi)=\lambda \chi_\varphi^{-1}, \quad \chi_\sigma=\lambda^{-1}\chi_\varphi, \quad \chi_\varphi=\chi,
$$
for two independent parameters $\lambda, \chi \in \mathbb{R}$ with $\lambda \chi>0$. In this manner, we have
$$
\mu=A\varPsi'(\varphi)-B\Delta\varphi-\chi\sigma,\qquad
\mathbf{q}_\sigma = -\nabla (\sigma+\lambda(1-\varphi)),
$$
which enable us to decouple the two mechanisms of chemotaxis and active transport. For instance, one may
switch off the effect of active transport by sending $\lambda\to 0$, while preserving the effect due to chemotaxis (see \cite[Section 3.3.3]{GLSS} for the case with a quasi-steady nutrient).

In the above setting, problem  \eqref{f3.c}--\eqref{ini0} can be written as follows
	\begin{subequations}
		\begin{alignat}{3}
		&\partial_t  \bm{ v}+\bm{ v} \cdot \nabla  \bm {v}-\mathrm{div} (2  \eta(\varphi) D\bm{v} )+\nabla p=(\mu+ \chi\sigma)\nabla \varphi,\label{1f3.c} \\
		&\mathrm{div}\ \bm{v}=0,\label{1f3.c1}\\
		&\partial_t \varphi+\bm{v} \cdot \nabla \varphi=\Delta \mu-\alpha(\varphi-c_0),\label{1f1.a} \\
		&\mu=A\varPsi'(\varphi)-B\Delta \varphi-\chi \sigma,\label{1f4.d} \\
		&\partial_t \sigma+\bm{v} \cdot \nabla \sigma= \Delta (\sigma+\lambda(1-\varphi))-\mathcal{C} h(\varphi) \sigma +S, \label{1f2.b}
		\end{alignat}
	\end{subequations}
subject to the boundary and initial conditions
\begin{alignat}{3}
&\bm{v}=\mathbf{0},\quad {\partial}_{\bm{n}}\varphi={\partial}_{\bm{n}}\mu={\partial}_{\bm{n}}\sigma=0,\qquad\qquad &\textrm{on}& \   \partial\Omega\times(0,T),
\label{1boundary}\\
&\bm{v}|_{t=0}=\bm{v}_{0}(x),\ \ \varphi|_{t=0}=\varphi_{0}(x), \ \ \sigma|_{t=0}=\sigma_{0}(x), \qquad &\textrm{in}&\ \Omega.
\label{1ini0}
\end{alignat}
Multiplying \eqref{1f3.c} with $\bm{v}$, \eqref{1f1.a} with $\mu$, \eqref{1f4.d} with $\partial_t \varphi$,\ \eqref{1f2.b} with $\lambda^{-1}\chi(\sigma+\lambda(1-\varphi))$, integrating over $\Omega$ and adding the resultants together, we obtain the following (formal) basic energy law:
\begin{align}
& \frac{d}{dt} \int_{\Omega} \Big[ \frac{1}{2}|\bm{v}|^2+A\varPsi(\varphi)+\frac{B}{2}|\nabla \varphi|^2+\frac{\chi}{2\lambda}|\sigma|^2+\chi\sigma(1-\varphi) \Big] dx \notag \\
& \qquad +\int_{\Omega} \Big[ 2\eta(\varphi)|D\bm{v}|^2 + |\nabla \mu|^2+ \frac{\chi}{\lambda}|\nabla(\sigma+\lambda(1-\varphi))|^2\Big] dx\nonumber\\
&\quad  =\int_\Omega \left[-\alpha(\varphi-c_0)\mu+\frac{\chi}{\lambda}(-\mathcal{C} h(\varphi) \sigma +S)(\sigma +\lambda(1-\varphi))\right] dx.
\label{BELb}
\end{align}
Thanks to \eqref{BELb}, one can check that the results obtained in \cite{H} and the current paper can be recovered for problem \eqref{1f3.c}--\eqref{1ini0} with only minor modifications.
\end{remark}
\begin{remark}
We note that there is another type of singular potential called the double obstacle potential that is commonly used in the literature. The corresponding free energy density takes the following form
$$
\varPsi_{\mathrm{obs}}(r)
=\frac{\theta_0}{2}(1-r^{2})+\mathbb{I}_{[-1,1]}(r)=
\begin{cases}
\displaystyle{\frac{\theta_{0}}{2}(1-r^2),}\quad\ \text{if}\ r\in [-1,1],\\
+\infty,\qquad \qquad \text{else},
\end{cases}
$$
and it can be obtained in the ``deep-quench limit" of the logarithmic potential \eqref{pot} by letting $\theta\to 0^+$ (see \cite{BE91}). Physically, this limit describes the dynamics of phase separation of binary mixtures, where the absolute temperature $\theta$ is far from the critical temperature $\theta_0$ below which phase separation occurs (see \cite{Abels11}).
In particular, in \cite{Abels11}, the author studied the deep-quench limit as $\theta \rightarrow 0^+$ (called ``double obstacle limit" therein) for the well-known ``Model H" describing the dynamics of a two-phase flow of viscous incompressible Newtonian fluids with the same density but different viscosity. For both two and three dimensions, he proved that global weak solutions of the approximate problem with the logarithmic potential \eqref{pot} converge as $\theta\to 0^+$ (up to a suitable subsequence) to global weak solutions of the corresponding Navier--Stokes--Cahn--Hilliard system, where the equation for the chemical potential is replaced by a differential inclusion related to the subgradient of $\varPsi_{\mathrm{obs}}$, that is, $
\mu+\Delta \varphi+\theta_0\varphi \in \partial I_{[-1,1]}(\varphi)$ (see \cite[Theorem 4.3]{Abels11}). We expect that a similar convergence result for global weak solutions to our problem \eqref{f3.c}--\eqref{ini0} with nutrient interaction can be obtained by extending the arguments in \cite{Abels11}. On the other hand, one can directly prove the existence of global weak solutions to the limit system with the double obstacle potential by applying suitable approximate schemes, for instance, via either the classical Moreau--Yosida regularization, or certain regular approximation of polynomial type (see e.g., \cite{EL}).
Finally, in analogy to \cite[Proposition 4.4, Theorem 4.5]{Abels11}, some further results on the uniqueness and regularity of weak solutions to the limit system can also be achieved.

\end{remark}

 \section{Strong Solutions: Global Well-posedness}\label{ws}
 \setcounter{equation}{0}
 \noindent In this section, we prove Theorem \ref{2main} on the existence and uniqueness of global strong solutions to problem \eqref{f3.c}--\eqref{ini0}. The proof of existence relies on a semi-Galerkin scheme introduced in \cite{H} and suitable \textit{a priori} higher-order energy estimates. By the approximate scheme we perform a finite dimensional approximation of the Navier--Stokes equation with an external force term (given by $\varphi, \mu, \sigma$), and then find an approximate solution $(\bm{v}^m, \varphi^m,\mu^m, \sigma^m)$ through a fixed point argument. Finally, we derive uniform estimates on the approximate solutions and pass to the limit as $m\to +\infty$ in the approximate formulation.

 \subsection{Estimate for an elliptic problem}
As a preliminary step, we report some useful estimates for the following Neumann problem of an elliptic equation with a singular nonlinearity:
\be
\begin{cases}
-B\Delta u+A\varPsi_{0}^{\prime}(u)=f,\quad \text { in } \Omega, \\
\partial_{\bm{n}} u=0, \qquad \qquad\qquad\ \ \,
\text { on } \partial \Omega,
\end{cases}
\label{ph1}
\ee
where the function $\varPsi_0$ satisfies the assumptions (H2)--(H3), and $A$, $B$ are given positive constants.

\bl\label{a.2} Suppose that  $\Omega\subset\mathbb{R}^{2}$ is a bounded domain with smooth boundary, the assumptions (H2)--(H3) are satisfied and $A,B>0$. For any $f\in L^2(\Omega)$, problem \eqref{ph1} admits a unique solution $u\in H_N^2(\Omega)$ that satisfies the equation a.e. in $\Omega$ and $\varPsi_{0}^{\prime}(u)\in L^2(\Omega)$. Besides,
if $f \in H^1(\Omega)$, then $\|\Delta u\| \leq C\|\nabla u\|^{\frac{1}{2}}\|\nabla f\|^{\frac{1}{2}}$. Moreover, for any $q\in [2,+\infty)$, there exists a positive constant $C=C(q,A,B,\Omega)$ such that
\begin{align}
\|u\|_{W^{2, q}}+\|\varPsi_{0}^{\prime}(u)\|_{L^{q}} &\leq C\left(1+\|f\|_{H^1}\right),\label{ph2}\\
\|\varPsi_{0}^{\prime \prime}(u)\|_{L^{q}}
&\leq C\Big(1+{e}^{C\|f\|_{H^1}^{2}}\Big).
\label{ph3}
\end{align}
\el
For the proof of Lemma \ref{a.2} we may refer to, for instance, \cite{A2009,CG,GGM2017,GGW,GMT} (where related results in three dimensions were also obtained).

As a consequence of the above lemma, we can obtain the following separation property for solutions to the elliptic problem \eqref{ph1}.
\bl\label{se}
Let the assumptions of Lemma \ref{a.2} be satisfied. For any $f \in H^1(\Omega)$, denote the corresponding solution to problem \eqref{ph1} by $u$. Then we have
$\varPsi_{0}^{\prime}(u)\in W^{1, q}(\Omega)$ for any  $q\in [2,+\infty)$. Moreover, there exists a constant $\delta\in (0,1)$ such that
\be
\|u\|_{L^{\infty}} \leq 1-\delta.\label{sep2}
\ee
\el
 \begin{proof}
 It follows from Lemma \ref{a.2} that $u\in W^{2,q}(\Omega)$, $\varPsi'_0(u)\in L^q(\Omega)$, for $q\in [2,+\infty)$. By the Sobolev embedding theorem, we also have $u\in C(\overline{\Omega})$. Next, for any integer $k>\epsilon_0^{-1}$, where $\epsilon_0$ is the constant given in (H2),  we define the following globally Lipschitz continuous function $h_{k}: \mathbb{R} \rightarrow \mathbb{R}$  (see e.g., \cite[Corollary 4.1]{GGM2017}):
 $$
 h_{k}(r)=\left\{\begin{array}{ll}
 1-k^{-1}, &\ \  r>1-k^{-1}, \\
 r, & \ \ r\in [-1+k^{-1},\ 1-k^{-1}], \\
 -1+k^{-1}, &\ \ r<-1+k^{-1}.
 \end{array}\right.
 $$
 Set $u_{k}=h_{k}\, \circ\, u$. Then we have $\partial_{x_i} u_k=I_{[-1+k^{-1},\,1-k^{-1}]}(u)\partial_{x_i} u$ and $|\partial_{x_i}u_k|\leq \partial_{x_i}u$, $i=1,2$. Moreover, $u_k\in W^{1,q}(\Omega)$. On the other hand, as $k\to +\infty$, we see that $\varPsi_0'(u_k)\to \varPsi_0'(u)$ and $\varPsi_0''(u_k)\to \varPsi_0''(u)$ a.e. in $\Omega$,  and it follows from the assumption (H2) that
 $ |\varPsi_{0}^{\prime}(u_k)|\le  |\varPsi_{0}^{\prime}(u)|$ and $|\varPsi_{0}^{\prime\prime}(u_k)|\le |\varPsi_{0}^{\prime\prime}(u)|$ a.e. in $\Omega$.
 For any $\omega\in C_0^{\infty}(\Omega )$, we deduce from Lebesgue's dominated convergence theorem that
 \begin{align}
 	\int_{\Omega} \varPsi_{0}^{\prime}(u) \partial_{x_i}\omega \, dx
 &=\int_{\Omega} \lim_{k\to+\infty}\varPsi_{0}^{\prime}(u_k) \partial_{x_i}\omega \, dx  =\lim_{k\to+\infty}\int_{\Omega} \varPsi_{0}^{\prime}(u_k) \partial_{x_i}\omega \, dx \nonumber \\
 	&=-\lim_{k\to+\infty}\int_{\Omega}  \varPsi_{0}^{\prime\prime}(u_k) \partial_{x_i}u_k\omega \, dx=-\int_{\Omega}  \lim_{k\to+\infty}\varPsi_{0}^{\prime\prime}(u_k) \partial_{x_i}u_k\omega \, dx\nonumber\\
 &=-\int_{\Omega}  \varPsi_{0}^{\prime\prime}(u) \partial_{x_i}u\omega \, dx,\quad i=1,2.
 \label{par0}
 \end{align}
 As a consequence, we see that
 $ \partial_{x_i}\varPsi_{0}^{\prime}(u) =\varPsi_{0}^{\prime\prime}(u) \partial_{x_i}u$ in the sense of distribution.
 By the estimates \eqref{ph2} and \eqref{ph3}, we have
 \be
 \|\varPsi_{0}^{\prime\prime}(u) \partial_{x_i}u\|_{L^q} \le\|\varPsi_{0}^{\prime\prime}(u)\|_{L^{2q}}\| \nabla u\|_{L^{2q}}\le C,\quad \forall\,q \in [2,+\infty),
 \ee
 where $C$ may depend on $\Omega$, $A$, $B$, $q$ and $\|f\|_{H^1}$. This together with \eqref{ph2} implies that
 $\varPsi_{0}^{\prime}(u) \in W^{1, q}(\Omega)$, for any $q\in [2,+\infty)$. Thus, from  the Sobolev embedding theorem in two dimensions, we can conclude
 $ \varPsi_{0}^{\prime}(u) \in L^{\infty}(\Omega)$,
 which combined with the assumption (H2) and the fact $u\in C(\overline{\Omega})$ yields the phase separation property \eqref{sep2}.
 \end{proof}

 \begin{remark}\label{regini}
  For the initial data considered in Theorem \ref{2main}: $\varphi_{0} \in H^{2}_N(\Omega)$, $\sigma_{0}\in H^1(\Omega)$ with $\left\|\varphi_{0}\right\|_{L^{\infty}} \leq 1,\ \left|\overline{\varphi}_{0}\right|<1$ and
	$\mu_{0}=A\varPsi^{\prime}\left(\varphi_{0}\right) -B\Delta \varphi_{0} -\chi\sigma_0\in H^1(\Omega)$, we infer from Lemma \ref{se} that there exists $\delta_0\in (0,1)$ such that
$\|\varphi_0\|_{L^{\infty}} \leq 1-\delta_0$, i.e., the initial phase function is indeed strictly separated from the pure states $\pm 1$. Then by the elliptic estimate for Neumann problem, we have $\varphi_0\in H^3(\Omega)$ (and thus the conclusion $\varphi\in C_w([0,T];H^3(\Omega))$ in Theorem \ref{2main} makes sense). Conversely, if $\varphi_0\in H^3(\Omega)$ and is strictly separate from $\pm 1$, it holds $\mu_0\in H^1(\Omega)$. The above observation is only valid in the two dimensional case, since the proof of the estimate  \eqref{ph3} in Lemma \ref{a.2} essentially relies on the Trudinger--Moser inequality in two dimensions.
 \end{remark}

\subsection{Semi-Galerkin scheme}
  We recall the Semi-Galerkin scheme introduced in \cite{H} (now stated in a strong form). Let the family $\{\bm{y}_{k}(x)\}_{k=1}^{\infty}$ be a basis of the Hilbert space $\bm{H}^1_{0,\mathrm{div}}(\Omega)$, which is given by eigenfunctions of the Stokes problem
\be
(\nabla \bm{y}_{k},\nabla \bm{w})=\lambda_{k}(\bm{y}_{k},\bm{w}),\quad  \forall\, \bm{w} \in {\bm{H}^1_{0,\mathrm{div}}(\Omega)},\quad \textrm{with}\ \|\bm{y}_{k}\|=1,\nonumber
\ee
where $\lambda_{k}$ is the eigenvalue corresponding to $\bm{y}_{k}$.  It is well-known that $0<\lambda_{1}< \lambda_{2}<... $ is an unbounded monotonically increasing sequence, $\{\bm{y}_{k}\}_{k=1}^{\infty}$ forms a complete orthonormal basis in $\bm{L}^2_{0,\mathrm{div}}(\Omega)$ and it is also orthogonal in $\bm{H}^1_{0,\mathrm{div}}(\Omega)$. If $\Omega$ is of class $C^{\infty}$, we can deduce from the elliptic regularity theory that $\bm{y}_{k} \in C^{\infty} $ for $k\in \mathbb{N}$.
For every $m\in \mathbb{N}$, we denote the finite-dimensional subspace of $\bm{H}^1_{0,\mathrm{div}}(\Omega)$ by
$\bm{H}_{m}:=\textrm{span} \{\bm{y}_{1}(x) ,...,\bm{y}_{m}(x)\}$.
Moreover, we use $\bm{P}_{\bm{H}_m}$ for the corresponding orthogonal projections from $\bm{L}^2_{0,\mathrm{div}}(\Omega)$ onto $\bm{H}_m$.

Let $T>0$ be given. For every $m\in \mathbb{N}$, we consider the following approximate problem: looking for functions
$\bm{v}^{m}(x,t):=\sum_{i=1}^{m}a_{i}^{m}(t)\bm{y}_{i}(x)$
and $(\varphi^{m},\, \mu^{m},\, \sigma^{m})$ satisfying
\begin{equation}
\begin{cases}
(\partial_t  \bm{ v}^{m},\bm{w})+( \bm{ v}^{m} \cdot \nabla  \bm {v}^{m},\bm{ w})+(  2\eta(\varphi^{m}) D\bm{v}^{m},D\bm{w})  =((\mu^{m}+\chi \sigma^{m})\nabla \varphi^{m},\bm {w}), \\
\qquad \qquad \text{for all } \bm{w} \in \bm{H}_{m}\ \text{and almost all } t\in (0,T),\\
\partial_t \varphi^m+ \bm{v}^m \cdot \nabla \varphi^m =\Delta  \mu^{m}-\alpha(\varphi^m-c_0), \qquad \qquad \qquad \qquad\ \ \, \text{a.e. in }\ \ \Omega\times(0,T),\\
\mu^{m}=A\varPsi'(\varphi^{m})-B\Delta \varphi^{m}-\chi \sigma^{m},  \qquad \qquad \qquad \qquad \qquad \, \qquad\text{a.e. in }\ \ \Omega\times(0,T),\\
\partial_t \sigma^m+\bm{v}^m \cdot \nabla \sigma^m =\Delta (\sigma^m+(1-\chi)\varphi^m)-\mathcal{C} h(\varphi^m) \sigma^m + S,  \quad \text{a.e. in }\ \ \Omega\times(0,T),
\end{cases}
\label{sapp1}
\end{equation}
as well as the boundary and initial conditions
\begin{equation}
\begin{cases}
\bm{v}^m=\mathbf{0},\quad {\partial}_{\bm{n}}\varphi^m ={\partial}_{\bm{n}}\mu^m={\partial}_{\bm{n}}\sigma^m=0, \qquad\qquad\quad \ \,  \textrm{on} \   \partial\Omega\times(0,T),
\\
\bm{v}^{m}|_{t=0}=\bm{P}_{\bm{H}_{m}} \bm{v}_{0},\quad \varphi^{m}|_{t=0}=\varphi_{0,},\quad  \sigma^{m}|_{t=0}=\sigma_{0},\quad \text{in}\ \Omega.
\end{cases}
\label{sapp2}
\end{equation}

Then we have the following result on the solvability of the approximate problem \eqref{sapp1}--\eqref{sapp2}:

\bp\label{p1} Let the assumptions in Theorem \ref{2main} be satisfied. Suppose that the initial data satisfy $ \bm{v}_{0} \in \bm{H}^1_{\mathrm{0,div}}(\Omega)$, $\varphi_{0} \in H^{2}_N(\Omega)$, $\sigma_{0}\in H^1(\Omega)$ with $\left\|\varphi_{0}\right\|_{L^{\infty}} \leq 1,\ \left|\overline{\varphi}_{0}\right|<1$ and
	$\mu_{0}=A\varPsi^{\prime}\left(\varphi_{0}\right) -B\Delta \varphi_{0} -\chi\sigma_0\in H^1(\Omega)$.
For every positive integer $m$, there exists a time $T_{m}\in [0,T]$ depending on $\bm{v}_{0},\ \varphi_{0},\ \sigma_{0}$, $\Omega$, $m$ and coefficients of the system such that the approximate problem \eqref{sapp1}--\eqref{sapp2} admits a strong solution $(\bm{v}^{m},\varphi^{m},\mu^{m},\sigma^{m})$ on $[0,T_{m}]\subset[0,T]$ with the following regularity
\begin{align}
&\bm{v}^{m} \in C^1([0,T_{m}];\bm{H}_m),\notag\\
& \varphi^{m} \in L^\infty(0,T_{m};W^{2,q}(\Omega)) \cap H^{1}(0,T_{m};H^1(\Omega)),\quad \forall\, q\in [2,+\infty),\notag \\
&\mu^{m} \in   L^\infty(0,T_{m};H^1(\Omega))\cap L^2(0,T_m;H^3(\Omega)),\notag \\
&\sigma^{m}   \in L^\infty(0,T_{m};H^1(\Omega))\cap L^{2}(0,T_{m};H^2(\Omega)) \cap H^1(0,T_{m};L^2(\Omega)).\notag \\
&\varphi^m \in L^{\infty}(\Omega\times(0,T_m))\ \textrm{and}\ |\varphi^m|<1\ \ \textrm{a.e.\ in}\ \Omega\times(0,T),\notag\\
& \varPsi_0^{\prime}(\varphi^m),\  \varPsi_0^{\prime\prime}(\varphi^m) \in L^{\infty}\left(0, T ;L^{q}(\Omega)\right),\quad \forall\, q\in[2,+\infty),\notag
\end{align}
that satisfies \eqref{sapp1} as well as the boundary and initial conditions \eqref{sapp2}. Moreover, there exists a constant $\delta_m\in (0,1)$ such that
\be
\|\varphi^m(t)\|_{L^{\infty}} \leq 1-\delta_m,\quad \forall\, t \in[0,T_m].
\label{sepm}
\ee
\ep
We only sketch the strategy for the proof of Proposition \ref{p1}, which is essentially parallel to that in \cite{H} for the case of weak solutions with suitable modifications (keeping in mind that we are now working with more regular initial data). For any given positive integer $m$, first, we take some $\widehat{\bm{u}}^m\in H^1(0,T;\bm{H}_m)$ and solve the equations for $(\varphi^m, \sigma^m)$ with convection terms related to the given velocity  $\widehat{\bm{u}}^m$ by extending the arguments in \cite{A2009} for the convective Cahn--Hilliard equation with singular potential (cf. also \cite{CG,GMT,GGW}). Then inserting the resulting strong solution $(\widehat{\varphi}^m, \widehat{\mu}^m, \widehat{\sigma}^m)$  back into the Galerkin approximation for $\bm{v}^m$, we obtain a solution $\widehat{\bm{v}}^m\in C^1([0,T_m]; \bm{H}_m)$ for some $T_m\in (0,T]$. After that, one can adapt the fixed point argument like in \cite[Section 3]{H} to show that the mapping $\widehat{\bm{u}}^m\mapsto \widehat{\bm{v}}^m$ has a fixed point $\bm{v}^m$ in a suitable closed convex set in $C([0,T_m];\bm{H}_m)$, which finally gives the existence of an approximate solution $(\bm{v}^{m},\varphi^{m},\mu^{m},\sigma^{m})$ on $[0,T_m]$ as stated in Proposition \ref{p1}. Thanks to Lemma \ref{se}, we further infer that $\varPsi_0^{\prime}(\varphi^m) \in L^{\infty}\left(0, T_m ;W^{1,q}(\Omega)\right)$ for any $q\in [2,+\infty)$ and the separation property \eqref{sepm} holds. Since the procedure is standard and lengthy, we omit the details here.

\subsection{\textit{A priori} estimates}\label{ap}
We proceed to perform \textit{a priori} estimates for the approximate solutions. To simplify the presentation, we drop the superscript $m$ and work with smooth solutions $(\bm{v},\varphi,\mu,\sigma)$ to problem \eqref{f3.c}--\eqref{ini0}.

\textbf{Lower order estimates}. Recalling  \cite[Section 3]{H}, from \eqref{mph1}, \eqref{BEL} we can derive the following lower order estimates:
\begin{align}
&\left \|  \bm{v}\right \|_{L^{\infty}(0,T;\bm{L}^2(\Omega))\cap L^{2}(0,T;\bm{H}^1(\Omega))\cap H^{1}(0,T;\bm{H}^1(\Omega)')} \notag\\
&\quad +\left \|  \varphi\right \|_{L^{\infty}(0,T;H^1(\Omega))\cap L^4(0,T;H^2(\Omega))\cap L^2(0,T;W^{2,q}(\Omega))\cap H^{1}(0,T;H^1(\Omega)')} \notag \\
& \quad +\left \|  \mu\right \|_{L^{2}(0,T;H^1(\Omega))} +\left \|  \sigma\right \|_{L^{\infty}(0,T;L^2(\Omega))\cap L^{2}(0,T;H^1(\Omega) )\cap H^{1}(0,T;H^1(\Omega)')} \notag\\
&\quad + \|\varPsi'(\varphi)\|_{L^2(0,T;L^q(\Omega))} \le  C_1,
\label{estimate1}
\end{align}
with $q\in [2,+\infty)$, where the constant $C_1>0$ depend on $\|\bm{v}_0\|$, $\|\varphi_0\|_{H^1}$, $\|\sigma_0\|$, $\int_\Omega \varPsi(\varphi_0)dx$, coefficients of the system, $\Omega$ and $T$. Besides, we have
\be
\|\varphi(t)\|_{L^{\infty}} \leq 1\quad \text{and}\quad |\overline{\varphi}(t)|<1,\quad \forall\, t\in [0,T].\label{L-ave}
\ee


\textbf{Estimate for $\|\nabla \bm{v}\|$}.
The estimates can be obtained as in  \cite[Section 4]{GMT}, with modifications for those terms involving $\sigma$. Below we only sketch the main steps. Multiplying equation \eqref{f3.c} by $\bm{S} \bm{v}$ and integrating over $\Omega$, we obtain
\begin{equation}
\frac{1}{2} \frac{d}{d t}\left\|\nabla\bm{v}\right\|^{2} +\left(\bm{v}\cdot\nabla\bm{v}, \bm{S}\bm{v}\right)-2\left(\operatorname{div}\left(\eta\left(\varphi\right) D\bm{v}\right), \bm{S}\bm{v}\right)=\left((\mu+\chi\sigma) \nabla \varphi, \bm{S}\bm{v}\right). \label{v3}
\end{equation}
Following the computations in the proof of \cite[Theorem 4.1]{GMT}, we obtain
\be
\begin{aligned}
-2\left(\operatorname{div}\left(\eta\left(\varphi\right) D \bm{v}\right), \bm{S}\bm{v}\right)
	&=-\left(\eta\left(\varphi\right) \Delta \bm{v}, \bm{S}\bm{v}\right)-2\left(\eta^{\prime}\left(\varphi\right) D \bm{v} \nabla \varphi, \bm{S}\bm{v}\right)\\
	&\geq \frac{5\eta_{*}}{6}\left\|\bm{S}\bm{v}\right\|^{2} -C\big(1+\left\|\varphi\right\|_{H^{2}}^{4}\big)\left\|\nabla \bm{v}\right\|^{2},
\end{aligned}\notag
\ee
\begin{equation}
\begin{aligned}
|\left((\bm{v}\cdot\nabla)\bm{v}, \bm{S}\bm{v}\right)|
& \leq\left\|\bm{v}\right\| _{\bm{L}^{4}}\left\|\nabla \bm{v}\right\| _{\bm{L}^{4}}\left\|\bm{S}\bm{v}\right\|
\leq \frac{\eta_{*}}{6}\left\|\bm{S}\bm{v}\right\|^{2} +C\left\|\nabla \bm{v}\right\|^{4},
\end{aligned}
\notag
\end{equation}
where the positive constant $C$ may depend on $\eta_*$, $\Omega$ and coefficients of the system. Next, in light of Young's inequality and the Sobolev embedding theorem, we have
\begin{equation}
\begin{aligned}
\left((\mu+\chi\sigma) \nabla \varphi, \bm{S} \bm{v}\right)
& \leq\left\|\mu+\chi\sigma\right\|_{L^{6}}\left\|\nabla \varphi\right\|_{\bm{L}^{3}} \left\|\bm{S} \bm{v}\right\| \\
& \leq \frac{\eta_{*}}{6}\left\|\bm{S} \bm{v}\right\|^{2} +C\left\|\varphi\right\|_{H^{2}}^{2} \big(1+\left\|\mu\right\|_{H^1}^{2}+\|\sigma\|_{H^1}^2\big).
\end{aligned}\notag
\end{equation}
Recalling the fact (see \cite{H})
\begin{equation*}
\begin{aligned}
&\| \varPsi'(\varphi)\|_{L^1}
\le C (1+\|  \nabla \mu\|),
\end{aligned}
\end{equation*}
where the constant $C$ depends on the lower order estimate $C_1$, we see that
$|\overline{\mu}|\leq C(1+\|\nabla \mu\|)$,
which together with the Poincar\'{e}--Wirtinger inequality implies
\begin{equation}
\left \|  \mu  \right \|_{H^1 }\le  C (1+\|  \nabla \mu\|).
\label{2amuL2H1}
\end{equation}
Combining the above estimates, we deduce from \eqref{v3} that
\begin{align}
\frac{1}{2} \frac{d}{d t}\left\|\nabla \bm{v}\right\|^{2}+\frac{\eta_{*}}{2}\left\|\bm{S} \bm{v}\right\|^{2}
& \leq C\left\|\nabla \bm{v}\right\|^{4}+ C(1+\|\varphi\|_{H^{2}}^{4})\left\|\nabla \bm{v}\right\|^{2}\notag\\
&\quad+C\|\varphi\|_{H^{2}}^{2}(1+\|\nabla \mu\|^{2}+\|\sigma\|_{H^1}^2).
\label{sv}
\end{align}
\noindent Multiplying equation \eqref{f3.c} by $\partial_{t} \bm{v}$ and integrating over $\Omega$, we have
\begin{equation}
\left\|\partial_{t}\bm{v}\right\|^{2} +\left(\bm{v}\cdot\nabla\bm{v}, \partial_{t}\bm{v}\right) -2\left(\operatorname{div}\left(\eta\left(\varphi\right) D\bm{v}\right), \partial_{t}\bm{v}\right)=\left((\mu+\chi\sigma) \nabla \varphi, \partial_{t}\bm{v}\right).
\label{v1}
\end{equation}
By Young's inequality and the Sobolev embedding theorem, we have
\begin{align}
\left((\mu+\chi\sigma) \nabla \varphi, \partial_{t} \bm{v}\right)
&  \leq \left\|\mu+\chi\sigma\right\|_{L^{6}}\left\|\nabla \varphi\right\|_{\bm{L}^{3}}\left\|\partial_{t} \bm{v}\right\| \notag\\
& \leq \frac{1}{6}\left\|\partial_{t} \bm{v}\right\|^{2}+C\left\|\varphi\right\|_{H^{2}}^{2} \big(1+\left\|\nabla \mu\right\|^{2}+\|\sigma\|_{H^1}^2\big). \notag
\end{align}
The other two terms on the left hand side of \eqref{v1} can be estimated exactly as in the proof of \cite[Theorem 4.1]{GMT} such that
\begin{equation}
\begin{aligned}
|\left(\bm{v}\cdot\nabla\bm{v}, \partial_{t}\bm{v}\right) | &  \leq \frac{1}{6}\left\|\partial_{t}\bm{v}\right\|^{2} +C\big(\left\|\bm{S}\bm{v}\right\|^{2} +\left\|\nabla\bm{v}\right\|^{4}\big),\\
2\left(\operatorname{div}\left(\eta\left(\varphi\right) D \bm{v}\right), \partial_{t} \bm{v}\right)
&\leq \frac{1}{6}\left\|\partial_{t} \bm{v}\right\|^{2}+C\big(\left\|\bm{S} \bm{v}\right\|^{2}+\left\|\varphi\right\|_{H^{2}}^{2}\left\|\nabla \bm{v}\right\|^{2}\big).
\end{aligned}\notag
\end{equation}
As a consequence,
\be
\begin{aligned}
	\left\|\partial_{t} \bm{v}\right\|^{2} &\leq C_2 \left\|\bm{S} \bm{v}\right\|^{2}+ C\left\|\nabla \bm{v}\right\|^{4} + C\left\|\varphi\right\|_{H^{2}}^{2} \big(1+\left\|\nabla \bm{v}\right\|^{2}+\left\|\nabla \mu\right\|^{2}+\|\sigma\|_{H^1}^2\big).\label{vt}
\end{aligned}
\ee
Multiplying \eqref{vt} by $\varpi=\frac{\eta_{*}}{4 C_{2}}>0$ and adding the resultant with \eqref{sv}, we find
\begin{equation}
\begin{aligned}
&\frac{1}{2} \frac{d}{d t}\|\nabla \bm{v}\|^{2} +\frac{\eta_{*}}{4}\|\bm{S} \bm{v}\|^{2} +\varpi\|\partial_{t} \bm{v}\|^{2}\\
& \quad \leq C\|\nabla \bm{v}\|^{4}+ C\big(1+\|\varphi\|_{H^{2}}^{4})(1+\|\nabla \bm{v}\|^{2}+\|\nabla \mu\|^{2}+\|\sigma\|_{H^1}^2\big).
\label{2v1}
\end{aligned}
\end{equation}

\textbf{Estimate for $\|\nabla \mu\|$}.
Multiplying equation \eqref{f1.a} by $\partial_{t} \mu$ and integrating over $\Omega$, we obtain
\begin{equation}
\frac{1}{2} \frac{d}{d t}\left\|\nabla \mu\right\|^{2}+\left( \partial_{t} \varphi,\partial_{t} \mu\right)+\left( \bm{v} \cdot \nabla \varphi,\partial_{t} \mu\right)+(\alpha(\varphi-c_0),\partial_{t} \mu)=0,\label{mu}
\end{equation}
As in \cite{MT}, we write
\be
\begin{aligned}\left( \bm{v} \cdot \nabla \varphi,\partial_{t} \mu\right) = \frac{d}{d t} \left(\bm{v} \cdot \nabla \varphi, \mu\right)  -\left(\partial_{t} \bm{v}\cdot \nabla \varphi, \mu\right)-\left(\bm{v} \cdot \nabla \partial_{t} \varphi, \mu\right). \end{aligned}\label{d1}
\ee
Then by Young's inequality, the Sobolev embedding theorem and \eqref{2amuL2H1}, we get
\begin{align}
(\partial_{t}\bm{v} \cdot \nabla \varphi, \mu) &\le\left\|\nabla \varphi\right\|_{\bm{L}^{3}}\left\|\partial_{t} \bm{v}\right\|\left\|\mu\right\|_{L^{6}} \leq \frac{\varpi}{4}\left\|\partial_{t} \bm{v}\right\|^{2} +C \left\|\varphi\right\|_{H^{2}}^{2} \big(1+\left\|\nabla \mu\right\|^{2} +\|\sigma\|_{H^1}^2\big),
\end{align}
and
\begin{equation}
\begin{aligned}\left( \bm{v}\cdot \nabla \partial_{t} \varphi,\mu\right) & \leq\left\|\bm{v}\right\|_{ \bm{L}^3}\left\|\nabla \partial_{t} \varphi\right\|\left\|\mu\right\|_{L^{6}} \leq \frac{B}{4}\left\|\nabla \partial_{t} \varphi\right\|^{2}+C \left\|\nabla \bm{v}\right\|^2\big(1+\left\|\nabla \mu\right\|^{2}+\|\sigma\|_{H^1}^2\big).
\end{aligned}\notag
\end{equation}
Next, since
\be
\int_\Omega [\partial_{t} \varphi(t)+\alpha(\varphi(t)-c_0)]dx=0,\quad \forall\, t \in[0, T],\label{mut}
\ee
then using the Poincar\'{e}--Wirtinger inequality, we deduce that
\begin{align}
|A(\theta-\theta_0)|\left\|\partial_{t} \varphi\right\|^{2}
&\leq C\left\|\partial_{t} \varphi+\alpha(\varphi-c_0)\right\|^{2} +C\left\|\alpha(\varphi-c_0)\right\|^{2} \notag\\
&\leq C \left\|\nabla( \partial_{t} \varphi+\alpha(\varphi-c_0)\right\|\left\|\nabla \mathcal{N} (\partial_{t} \varphi+ \alpha(\varphi-c_0))\right\| +C\left\|\alpha(\varphi-c_0)\right\|^{2}\notag\\
&\leq\frac{B}{16} \left\|\nabla( \partial_{t} \varphi+\alpha(\varphi-c_0)\right\|^{2} +\frac{16C^2}{B}\left\|\nabla \mathcal{N} (\partial_{t} \varphi+ \alpha(\varphi-c_0))\right\|^2\notag\\
&\quad+C\left\|\alpha(\varphi-c_0)\right\|^{2}\notag\\
&\leq\frac{B}{8} \left\|\nabla \partial_{t} \varphi\right\|^{2} +C \big(1+\left\|\partial_{t} \varphi \right\|_{(H^1)'}^{2} +\left\| \varphi\right\|_{H^1}^{2}\big),\notag
\end{align}
and
\begin{align}
|\chi\left(  \partial_{t} \varphi,\partial_{t} \sigma\right)|
&=|\chi||\left( \partial_{t} \varphi+\alpha(\varphi-c_0) ,\partial_{t} \sigma\right)-\left( \alpha(\varphi-c_0),\partial_{t} \sigma\right)|\notag\\
&\leq C\left\| \partial_{t} \varphi+\alpha(\varphi-c_0)\right\|_{H^1}\left\| \partial_{t} \sigma\right\|_{(H^1)'}+C\left\|\alpha(\varphi-c_0)\right\|_{H^1}\left\| \partial_{t} \sigma\right\|_{(H^1)'}\notag\\
&\leq \frac{B}{8} \left\|\nabla \partial_{t} \varphi\right\|^{2} +C \big(1+\left\| \partial_{t} \sigma\right\|_{(H^1)'}^2 +\left\| \varphi\right\|_{H^1}^{2}\big).
\notag
\end{align}
Finally, from the assumption  (H2) and the above two estimate, we obtain
\begin{align}\left( \partial_{t} \varphi,\partial_{t} \mu\right)
&=B\left\|\nabla \partial_{t} \varphi\right\|^{2} +A\left(\partial_{t} \varphi,\varPsi''\left(\varphi\right) \partial_{t} \varphi\right) -\chi\left(  \partial_{t} \varphi,\partial_{t} \sigma\right) \notag\\
& \geq B\left\|\nabla \partial_{t} \varphi\right\|^{2}-|A(\theta-\theta_0)|\left\|\partial_{t} \varphi\right\|^{2} -|\chi(\partial_{t} \varphi,\, \partial_{t} \sigma)|\notag\\
& \geq \frac{3B}{4}\left\|\nabla \partial_{t} \varphi\right\|^{2} -C \big(1+\left\|\partial_{t} \varphi\right\|_{(H^1)'}^{2} +\left\| \partial_{t} \sigma\right\|_{(H^1)'}^2 +\left\| \varphi\right\|_{H^1}^{2}\big).
\notag
\end{align}
Combining the above estimates, we deduce from  \eqref{mu} that
\begin{align}
&\frac{d}{d t}\Big(\frac{1}{2}\|\nabla \mu\|^{2}+(\bm{v} \cdot \nabla \varphi, \mu)\Big) +\frac{B}{2}\|\nabla \partial_{t} \varphi\|^{2} \notag\\
&\quad \leq \frac{\varpi}{4}\left\|\partial_{t} \bm{v}\right\|^{2} + C \big(1+\left\|\partial_{t} \varphi\right\|_{(H^1)'}^{2}+ \left\| \partial_{t} \sigma\right\|_{(H^1)'}^2 +\left\| \varphi\right\|_{H^1}^{2}\big) \notag\\
&\qquad  +C \big(1+\|\varphi\|_{H^{2}}^{2}+\|\nabla\boldsymbol{v}\|^{2}\big) \big(1+\|\nabla \mu\|^{2}+\|\sigma\|_{H^1}^2\big).
\label{22mu}
\end{align}


\textbf{Estimate for $\|\nabla \sigma\|$}. Multiplying \eqref{f2.b} by $-\Delta\sigma$ and integrating over $\Omega$, we get
\be
\frac{1}{2} \frac{d}{d t}\left\|\nabla\sigma\right\|^{2}-(\bm{v} \cdot \nabla \sigma,\Delta\sigma)+\|\Delta \sigma\|^2= \chi (\Delta \varphi,\Delta\sigma)+\mathcal{C}( h(\varphi) \sigma,\Delta\sigma) - (S,\Delta\sigma ).
\ee
By the Gagliardo--Nirenberg inequality and Young's inequalities, we see that
\begin{align}
|(\bm{v} \cdot \nabla \sigma,\Delta\sigma)|
&\le \|\bm{v}\|_{\bm{L}^{4}}\| \nabla\sigma\|_{\bm{L}^{4}}\|\Delta\sigma\|\notag\\
&\le\frac{1}{8}\|\Delta\sigma\|^2 + C\|\bm{S}\bm{v}\|^{\frac{1}{2}} \|\bm{v}\|^{\frac{3}{2}}\|\sigma\|_{H^2} \|\sigma\|_{H^1}\notag\\
&\leq \frac{1}{8}\|\Delta\sigma\|^2 + C\|\bm{S}\bm{v}\|^{\frac{1}{2}} \|\bm{v}\|^{\frac{3}{2}}(\|\Delta \sigma\|+\|\sigma\|)\|\sigma\|_{H^1} \notag\\
&\le\frac{1}{4}\|\Delta\sigma\|^2 +\frac{\eta_*}{8}\|\bm{S}\bm{v}\|^{2} +
C\big(\|\bm{v}\|^{6}\|\sigma\|_{H^1}^{4} +\|\bm{v}\|^{6}+\|\sigma\|_{H^1}^{4}\big).\notag
\end{align}
The remaining terms can be easily handled by using Young's inequality
\begin{align*}
|\chi (\Delta \varphi,\Delta\sigma)+\mathcal{C}( h(\varphi) \sigma,\Delta\sigma) - (S,\Delta\sigma ) | \leq \frac14\|\Delta \sigma\|^2+C\big(\|\Delta\varphi\|^2+\|\sigma\|^2+\|S\|^2\big).
\end{align*}
Hence, we have
\begin{align}
&\frac{1}{2} \frac{d}{d t}\left\|\nabla\sigma\right\|^{2}+\frac{1}{2}\|\Delta \sigma\|^2 \notag\\
&\quad \leq \frac{\eta_*}{8}\|\bm{S}\bm{v}\|^{2}+ C \big(\left\|\Delta\varphi\right\|^{2} + \|\bm{v}\|^{6}\|\sigma\|_{H^1}^{4} +\|\bm{v}\|^{6} +\|\sigma\|_{H^1}^{4}+\left\|S\right\|^{2}\big).
\label{2sig}
\end{align}

\textbf{Higher order estimate}. Adding  \eqref{2v1}, \eqref{22mu} and \eqref{2sig} together, and making use of the lower order estimate \eqref{estimate1}, we arrive at
\be
\frac{d}{d t}\Lambda_1(t)  +  \mathcal{M}(t)  \leq    \mathcal{R}_{1}(t)\Lambda_2(t)+\mathcal{R}_{2}(t),\quad   \forall\,t\in[0,T].
\label{2menergy}
\ee
with
\begin{align*}
&\Lambda_1(t)=\frac12\|\nabla \bm{v}\|^{2}+\frac{1}{2}\|\nabla \mu\|^{2}+(\bm{v} \cdot \nabla \varphi, \mu)+\frac{1}{2}\left\|\nabla\sigma\right\|^{2},
\notag\\
&\mathcal{M}(t) =\frac{\eta_{*}}{8}\|\bm{S} \bm{v}\|^{2}+\frac{3\varpi}{4}\|\partial_{t} \bm{v}\|^{2} +\frac{B}{2}\|\nabla \partial_{t} \varphi\|^{2} +\frac{1}{2}\|\Delta \sigma\|^2,\\
&\mathcal{R}_{1}(t) = C\big(1+\|\nabla \bm{v}\|^{2}+\|\varphi\|_{H^{2}}^{4}+\|\nabla \sigma\|^2\big),\\
&\Lambda_2(t)=1+\|\nabla \bm{v}\|^{2}+\|\nabla \mu\|^{2}+\|\nabla\sigma\|^2,\\
&\mathcal{R}_{2}(t) =C\big(1+\left\|\partial_{t} \varphi\right\|_{(H^1)'}^{2} +\left\| \partial_{t} \sigma\right\|_{(H^1)'}^2 + \left\|\Delta  \varphi\right\|^{2}+ \left\|S\right\|^{2} \big),
\end{align*}
where the constant $C$ may depend on $\|\bm{v}_0\|$, $\|\varphi_0\|_{H^1}$, $\|\sigma_0\|$, $\int_\Omega \varPsi(\varphi_0)dx$, coefficients of the system, $\Omega$ and $T$. Recalling the proof of \cite[Theorem 4.1]{GMT}, we also have
\be\begin{aligned}
	|(\bm{v} \cdot \nabla \varphi, \mu)|
&=|(\bm{v}, \varphi \nabla  \mu)| \leq C\left\|\bm{v}\right\|_{\bm{L}^4}^2\|\varphi\|^2_{L^4}+\frac{1}{4}\left\|\nabla \mu\right\|^{2}\\
	&\leq\frac{1}{4}\left\|\nabla \bm{v}\right\|^2+\frac{1}{4}\left\|\nabla \mu\right\|^{2}+C_3,\notag
\end{aligned}
\ee
where the constant $C_3$ depends on $C_1$ in \eqref{estimate1}. Then we have
\be
\Lambda_1+(1+C_{3})\geq \frac{1}{4}\Lambda_2\geq \frac{1}{4}\big(\Lambda_1-C_{3}\big).
\label{equiv}
\ee
Integrating \eqref{2menergy} with respect to time and using \eqref{equiv}, we arrive at
$$
\Lambda_1(t)+\int_0^t\mathcal{M}(s) ds \leq \Lambda_1(0)+ 4\int_0^t\mathcal{R}_1(s)\Lambda_1(s)ds +C_4\int_0^t\big(\mathcal{R}_1(s)+\mathcal{R}_2(s)\big)ds.
$$
Keeping in mind the facts $\mathcal{R}_{1}(t)$, $\mathcal{R}_{2}(t)\in L^1(0,T)$ due to \eqref{estimate1}, by a suitable version of Gronwall's lemma (see e.g., \cite[Lemma 3.1]{GL17e}), we obtain
\begin{align}
&\Lambda_1(t) +\int_{0}^t \mathcal{M}(s)\,ds
\leq \Lambda_3(t) +4\int_{0}^t \left[ \Lambda_3(s)\mathcal{R}_{1}(s) \exp\Big(4\int_{s}^t\mathcal{R}_{1}(r)\,dr\Big)\right]ds,\quad \forall\,t\in [0,T],
\notag
\end{align}
where
$$
\Lambda_3(t)=\Lambda_1(0)+C_4\int_{0}^t \big(\mathcal{R}_{1}(s)+\mathcal{R}_{2}(s)\big)\, ds.
$$
The above inequality together with \eqref{estimate1} and \eqref{2amuL2H1} immediately yields the following higher order estimates
\begin{align}
&\sup_{t \in[0, T]}\left(\| \bm{v}(t)\|_{\bm{H}^1}^{2}+\|\mu(t)\|_{H^1}^{2}+\|\sigma(t)\|_{H^1}^{2}\right) \leq C,\label{muphi}
\\
&\int_{0}^{t} \big(\|\bm{S} \bm{v}(s)\|^{2}+\|\partial_{t} \bm{v}(s)\|^{2} +\|\nabla \partial_{t} \varphi(s)\|^{2} +\| \sigma(s)\|_{H^2}^2\big) \, ds\leq C,\quad \forall\,t\in[0,T],\label{phih1}
\end{align}
where the positive constant $C$ depends on $\|\bm{v}_0\|_{\bm{H}^1}$, $\|\mu_0\|_{H^1}$, $\|\sigma_0\|_{H^1}$, $\|\varphi_0\|_{H^1}$, $\int_\Omega \varPsi(\varphi_0)dx$, coefficients of the system, $\Omega$ and $T$. Moreover, let us consider the elliptic problem for $\varphi$:
\be
\begin{cases}
-B\Delta \varphi +A\varPsi_{0}^{\prime}(\varphi)=\mu+\theta_0\varphi+\chi\sigma,\quad \text { in } \Omega, \\
\partial_{\bm{n}} \varphi =0, \qquad \qquad \qquad \qquad \qquad\quad\,
\text { on } \partial \Omega.
\end{cases}
\label{ellsol}
\ee
Then  from \eqref{muphi}--\eqref{phih1} and Lemma \ref{a.2} we infer that for any $q\in [2,+\infty)$,
\begin{align}
\|\varphi\|_{L^\infty(0,T;W^{2, q}(\Omega))} +\|\varPsi^{\prime}(\varphi)\|_{L^\infty(0,T; L^{q}(\Omega))} + \|\varPsi^{\prime \prime}(\varphi)\|_{L^\infty(0,T;L^{q}(\Omega))}
\leq C.
\label{psis2}
\end{align}

 \textbf{Estimates on time derivatives}.
 From \eqref{mut}, \eqref{phih1} and the Poincar\'{e}--Wirtinger inequality, we see that
\begin{equation}
\left \|  \partial_{t}\varphi\right \|_{L^{2}(0,T;H^1(\Omega))} \le C.
\label{phimt1}
\end{equation}
On the other hand, a similar argument like in  \cite[(4.15)--(4.18)]{LW} leads to
\begin{align}
& \left \| \partial_{t}\bm{v}\right \|_{ L^{2}(0,T;\bm{L}^2(\Omega))} \le C,\quad \left \| \partial_{t}\sigma\right \|_{ L^2(0,T;L^2(\Omega))} \le C. \label{vsigmt2d}
\end{align}

\textbf{Some further estimates for $\mu$}. It follows from \eqref{muphi}, \eqref{phih1}, \eqref{psis2} and \eqref{phimt1} that
\begin{equation}
\begin{aligned}
& \int_{0}^T \left \|\bm{v}(t)\cdot \nabla\varphi(t)\right \|_{H^{1}}^2\,dt\\
&\quad \leq \|\bm{v}\|_{L^{\infty}(0,T;\bm{L}^4(\Omega))}^2 \int_{0}^T \|  \varphi(t)\|_{W^{2,4}}^2 \ dt+ \|\varphi\|_{L^{\infty}(0,T;W^{1,4}(\Omega))}^2\int_{0}^T\|\bm{v}(t)\|_{\bm{W}^{1,4}}^2 \ dt\\
&\quad \le C,
\end{aligned}
\nonumber
\end{equation}
which yields $\partial_{t} \varphi+\boldsymbol{v} \cdot \nabla \varphi+\alpha(\varphi-c_0) \in L^{2}(0, T ; H^{1}(\Omega))$.
Then from equation \eqref{f1.a} and the classical elliptic estimate for $\mu$, we conclude that
\begin{align}
\mu \in L^{2}\left(0, T ; H^{3}(\Omega)\right).
\label{muL2H3}
\end{align}

\subsection{Proof of Theorem \ref{2main}}

The formal \textit{a priori} estimates obtained in the previous section can be justified by the approximate solutions $(\bm{v}^m, \varphi^m, \mu^m, \sigma^m)$ obtained in Proposition \ref{p1}. It is straightforward to check that the constants in those estimates are independent of the index $m$ as well as the existence time $T_m$. Hence, we are able to pass to the limit as $m\to+\infty$ (up to a subsequence) and obtain the existence of a global strong solution $(\bm{v}, \varphi, \mu, \sigma)$ to problem \eqref{f3.c}--\eqref{ini0} on $[0,T]$ (cf. \cite{H} for the case of global weak solutions). Finally, the pressure $p$ can be recovered as  in \cite{LW} (see \cite{S}). This procedure is standard, so we omit the details here.

Since the strong solution  can also be viewed as a weak one in Definition \ref{maind}, its uniqueness (as long as it exists) is then an immediate consequence of the continuous dependence estimate for weak solutions (in certain weaker norms) obtained in \cite[Theorem 2.2]{H}.

Finally, we mention that some additional regularity of the strong solution $(\bm{v}, \varphi, \mu, \sigma)$ can be obtained once the separation property of $\varphi$ is proven (i.e., \eqref{sep4}), see \eqref{phih3}--\eqref{conti} in Remark \ref{regh1} below.

The proof of Theorem \ref{2main} is complete.
\hfill $\square$

 \section{Strong Solutions: Separation Property}
 \label{phs} \setcounter{equation}{0}

 In this section, we study the separation property of the phase function $\varphi$ and some further regularity of the strong solutions for positive time.

 \subsection{Strict separation property on $[0,T]$}

 \textbf{Proof of Theorem \ref{phase}, Part I}.
Let $(\bm{v},\varphi, \mu, \sigma)$ be the global strong solution to problem \eqref{f3.c}--\eqref{ini0} on $[0,T]$ obtained in Theorem \ref{2main}. Then consider again the elliptic problem \eqref{ellsol} and apply Lemma \ref{se}, we infer from \eqref{muphi} that
 $\varPsi_0^{\prime}(\varphi) \in L^{\infty}\left(0, T ;W^{1,q}(\Omega)\right)$ for any $q\in[2,+\infty)$,
 and there exists a constant $\delta_T\in (0,1)$ such that
 \be
 \|\varphi(t)\|_{L^{\infty}} \leq 1-\delta_T,\quad \forall\, t \in[0,T].
 \label{sep4}
 \ee
This completes the proof for the first conclusion in  Theorem \ref{phase}. \hfill $\square$

\begin{remark}\label{regh1}
The strict separation property \eqref{sep4} enables us to say more about the regularity of strong solutions on $[0,T]$. Since $\varPsi\in C^3[-1+\delta_T,\, 1-\delta_T]$, we  can deduce from \eqref{psis2} and \eqref{sep4} that $\varPsi' (\varphi )\in L^\infty (0 ,T ;H^2(\Omega ))$. Now applying the classical elliptic estimate to the Neumann problem \eqref{ellsol}, we get
 \be
 \varphi \in L^\infty (0 ,T ;H^3 (\Omega )),\quad \varphi\in L^2(0,T; H^4(\Omega)).
 \label{phih3}
 \ee
By a similar approximate argument like in \cite[Theorem 4.1]{GMT}, using the separation property \eqref{sep4} and estimates \eqref{phimt1}--\eqref{vsigmt2d}, we can infer from \eqref{f4.d}, that for any $\zeta\in H^1(\Omega)$, the following term
\begin{align*}
|\langle\partial_t \mu, \zeta\rangle_{(H^1)',H^1}|
&\leq |B(\nabla \partial_t \varphi, \nabla \zeta)|+|A(\varPsi''(\varphi)\partial_t\varphi, \zeta)|+|\chi (\partial_t\sigma, \zeta)|\\
&\leq B\|\nabla \partial_t \varphi\|\|\nabla \zeta\|+A\max_{s\in[-1+\delta_T,\,1-\delta_T]}|\varPsi''(s)| \|\partial_t\varphi\|\|\zeta\| +|\chi|\|\partial_t\sigma\|\|\zeta\|
\end{align*}
is bounded, which implies
\be
\partial_{t} \mu\in L^{2}\left(0, T ; H^{1}(\Omega)'\right).
\label{muth1}
\ee
Note that here we do not need the assumption on the convexity of $\varPsi''$ as required in \cite[Section 2.2]{GMT}.
Then by the Aubin--Lions--Simon lemma (see \cite{si87}), we can conclude the time continuity property of strong solutions from \eqref{phih1}, \eqref{phimt1}, \eqref{vsigmt2d}, \eqref{muL2H3}, \eqref{phih3} and \eqref{muth1} such that
\begin{equation}
\begin{aligned}
&\bm{v}\in C([0,T]; \bm{H}_{0,\mathrm{div}}^1(\Omega)),\quad \varphi\in C_w([0,T]; H^3(\Omega)),\\
&\mu\in C([0,T];H^1(\Omega)),\qquad\   \sigma\in C([0,T];H^1(\Omega)).
\end{aligned}
\label{conti}
\end{equation}
\end{remark}

 \subsection{Uniform separation property on $[0,+\infty)$}

 \begin{lemma}
 Assume that the additional assumption \eqref{ass1} is satisfied. Then the strong solution to problem \eqref{f3.c}--\eqref{ini0} exists on  $[0,+\infty)$ and satisfies the following uniform-in-time estimates:
 \begin{align}
 &\|\bm{v}(t)\|^2 +\|\varphi(t)\|_{H^1}^2+\|\sigma(t)\|^2 \notag \\
 &\quad +\int_{t}^{t+1}\big(\|\nabla\bm{v}(s)\|^2
 +\|\varphi(s)\|_{H^2}^4 + \|\nabla \mu(s)\|^2+ \|\nabla \sigma(s)\|^2\big)\,  d s \le C,\quad \forall\, t\geq 0,
 \label{uni1}\\
 &\int_{t}^{t+1}\big( \left\|\partial_{t} \varphi(s)\right\|_{(H^1)'}^{2} +\left\| \partial_{t} \sigma(s)\right\|_{(H^1)'}^2\big)\,ds
 \leq C,\quad \forall\, t\geq 0,\label{uni2}
 \end{align}
 where the constant $C$ depends on $\|\bm{v}_0\|_{\bm{H}^1}$, $\|\mu_0\|_{H^1}$, $\|\sigma_0\|_{H^1}$, $\|\varphi_0\|_{H^1}$, $\int_\Omega \varPsi(\varphi_0)dx$, coefficients of the system and $\Omega$, but not on $t$.
 \end{lemma}
 \begin{proof}
The existence and uniqueness of a global strong solution $(\bm{v}, \varphi, \mu, \sigma)$ on an arbitrary time interval $[0,T]$ are guaranteed by Theorem \ref{2main}. Thus, we only need to derive \textit{a priori} estimates that are uniform in time.

We keep in mind that
 \be
 \|\varphi(t)\|_{L^\infty}\leq 1,\quad t\geq 0.\label{inif}
 \ee
 Next, integrating \eqref{f2.b} over $\Omega$, we get
 \be
 \frac{d}{dt}\int_\Omega \sigma dx =\int_\Omega S(x,t)\, dx,\notag
 \ee
 so that
 \be
 \overline{\sigma}(t)=\overline{\sigma_{0}}+\int_0^t \overline{S(s)}ds,\quad \forall\, t\geq 0.
 \label{averphie}
 \ee
 Then it follows from the Poincar\'{e}--Wirtinger inequality that
 \be
 \|\sigma\|\leq C_P(\|\nabla \sigma\|+|\overline{\sigma}|)\leq C_P\|\nabla \sigma\|+C,
 \label{sigmaa}
 \ee
 where $C_P$ depends on $\Omega$ and $C$ depends on $\overline{\sigma_0}$, $S_0$, $\Omega$.

  Multiplying \eqref{f3.c} by $\bm{v}$, \eqref{f1.a} by $\mu$, \eqref{f4.d} by  $\partial_t \varphi$,\ \eqref{f2.b} by $\sigma+\chi(1-\varphi)$, integrating over $\Omega$ and adding the resultants together, we obtain (i.e., taking $\mathcal{C}=0$ in \eqref{BEL})
 \begin{align}
 \frac{d}{dt}\mathcal{E}(t)
 +\mathcal{D}(t)+\int_\Omega \alpha(\varphi-c_0)\mu \, dx =   \int_\Omega S(\sigma +\chi(1-\varphi))\, dx,
 \label{BEL2}
 \end{align}
 where
 \begin{align}
 \mathcal{E}(t)&=\frac{1}{2}\|\bm{v}(t)\|^2
 + \int_\Omega A\varPsi(\varphi(t))\, dx + \frac{B}{2}\|\nabla \varphi(t)\|^2
 +\frac{1}{2}\|\sigma(t)\|^2\notag\\
 &\quad
 +\int_\Omega \chi\sigma(t)(1-\varphi(t))\, dx,\label{E}\\
 \mathcal{D}(t)& =\int_\Omega 2\eta(\varphi(t))|D\bm{v}(t)|^2\, dx + \|\nabla \mu(t)\|^2+\|\nabla(\sigma(t)+\chi(1-\varphi(t))\|^2.\label{D}
 \end{align}
 Using the convexity of $\varPsi_0$ and the estimate \eqref{inif}, we can treat the third term on the left-hand side of \eqref{BEL2} as (cf. \cite{GGM2017, H})
 \begin{align*}
 \alpha\int_\Omega(\varphi-c_0)\mu\,dx
 &= \alpha B\|\nabla \varphi\|^2 + \alpha A\int_\Omega \varPsi'(\varphi)(\varphi-c_0)\,dx - \alpha\chi\int_\Omega \sigma(\varphi-c_0)\,dx\notag\\
 &\geq \alpha B\|\nabla \varphi\|^2
 + \alpha A\int_\Omega \varPsi_0(\varphi)\,dx
-\alpha A\int_\Omega \varPsi_0(c_0)\,dx \notag\\
 &\quad
-\alpha A\theta_0\int_\Omega\varphi(\varphi-c_0)\,dx
 - \alpha\chi\int_\Omega \sigma(\varphi-c_0)\,dx\notag\\
 & \geq \alpha\Big(\frac{ B}{2}\|\nabla \varphi\|^2+ A\int_\Omega \varPsi(\varphi)\,dx\Big)-\alpha \xi \|\sigma\|^2 - C\alpha(1+\xi^{-1}),
 \end{align*}
 where the positive constant $C$ may depend on $A$, $\varPsi_0$, $c_0$, $\theta_0$, $\chi$, $\Omega$, but is independent of the initial data, while $\xi\in (0,1)$ is a small parameter to be determined later. Next, using the Cauchy--Schwarz inequality, Young's inequality and \eqref{inif}, we easily see that
 \begin{align*}
 \int_\Omega S(\sigma +\chi(1-\varphi)) dx
 \leq \xi \|\sigma\|^2+\xi^{-1}\|S\|^2+C,
 \end{align*}
 where $\xi\in (0,1)$, the positive constant $C$ may depend on $\chi$, $\Omega$, but is independent of the initial data. From \eqref{inif} and the assumption (H2), we see that  $|\int_\Omega \varPsi(\varphi(t))dx|\leq |\Omega|\max_{s\in[-1,1]}|\varPsi(s)|$. On the other hand,  integration by parts and Young's inequality yield
 $$
 \|\nabla \varphi\|^2\leq \|\Delta \varphi\|\|\varphi\|\leq \xi\|\Delta\varphi\|^2+\xi^{-1}|\Omega|,\quad \forall\, \xi\in (0,1).
 $$
 Then we infer from the above estimates and \eqref{BEL2} the following inequality
 \begin{align}
 &\frac{d}{dt}\mathcal{E}(t)
 +\mathcal{D}(t) + (1+\alpha)\Big(\frac{ B}{2}\|\nabla \varphi\|^2+ A\int_\Omega \varPsi(\varphi)\,dx \Big) +\frac{\xi}{2}\|\sigma(t)\|^2 \notag\\
 &\qquad +\xi \int_\Omega \chi\sigma(t)(1-\varphi(t))\, dx \notag\\
 &\quad  \leq  2\xi(1+\alpha)C_P^2 \|\nabla \sigma\|^2 +\frac{\xi(1+\alpha) B}{2}\|\Delta \varphi\|^2 +\xi^{-1}\|S\|^2+C(1+\alpha)(1+\xi^{-1}).
 \label{BEL3}
 \end{align}
 Note note that here we allow $\alpha=0$. Next, multiplying \eqref{f4.d} by $\varphi$ and integrating over $\Omega$, we get
 \begin{align}
 &\frac12 \frac{d}{dt}\| \varphi\|^2
 +B\|\Delta\varphi\|^2 + \alpha \|\varphi\|^2
 =A \int_{\Omega}\varPsi'(\varphi)\Delta\varphi dx -\chi\int_\Omega \sigma\Delta\varphi dx +\alpha c_0\int_\Omega  \varphi dx.
 \label{energy e}
 \end{align}
 The first and second terms on the right-hand side of \eqref{energy e} can be estimated as follows:
 \begin{align}
 A\int_{\Omega}\varPsi'(\varphi)\Delta\varphi \ dx
 &=- A\int_{\Omega}\varPsi_{0}''(\varphi)|\nabla\varphi|^2 \, dx
 -A\theta_{0}\int_{\Omega}\varphi\Delta\varphi\ dx\notag\\
 &\le \frac{B}{8}\|\Delta\varphi\|^2  +\frac{2A^2\theta_0^2}{B}\|\varphi\|^2.
 \notag
 \end{align}
 \begin{align}
 -\chi \int_{\Omega} \sigma \Delta\varphi dx
 &=\chi\int_{\Omega}\nabla\varphi\cdot \nabla\sigma \, dx \leq \chi^2\|\nabla\varphi\|^2+\frac{1}{4}\|\nabla\sigma\|^2 \notag\\
 & \le \frac{B}{8}\|\Delta\varphi\|^2  +\frac{1}{4}\|\nabla\sigma\|^2 +\frac{2\chi^4}{B}\|\varphi\|^2,
 \label{diss2}
 \end{align}
 while the last term can be simply treated as
 $ \alpha c_0\int_\Omega  \varphi dx\leq \alpha |\Omega|$.
 Inserting the above estimates into \eqref{energy e}, we obtain
 \begin{align}
 \frac12 \frac{d}{dt} \| \varphi\|^2
 +\frac{3B}{4}\|\Delta\varphi\|^2 \leq \frac{1}{4}\|\nabla\sigma\|^2 +C,
 \label{energy e1}
 \end{align}
 where $C$ only depends on $A$, $B$, $\theta_0$, $\chi$, $\alpha$ and $\Omega$. Now multiplying \eqref{f2.b} by $\sigma$ and integrating over $\Omega$, we get
 \begin{align}
 \frac12 \frac{d}{dt}\|\sigma\|^2
 +\|\nabla \sigma\|^2
 =-\chi\int_{\Omega} \sigma \Delta\varphi  \ dx +\int_\Omega S\sigma dx.
 \notag
 \end{align}
 In light of \eqref{diss2}, we get
 \begin{align}
 \frac12 \frac{d}{dt}\|\sigma\|^2
 +\frac34\|\nabla \sigma\|^2
 \leq \frac{B}{4}\|\Delta\varphi\|^2 +\xi\|\sigma\|^2 +\xi^{-1}\|S\|^2 +C.
 \label{energy e3}
 \end{align}
 Adding \eqref{energy e1} with \eqref{energy e3} and using \eqref{sigmaa}, we see that
 \begin{align}
 \frac12 \frac{d}{dt}\big(\| \varphi\|^2+\|\sigma\|^2\big)
 +\frac{B}{2}\|\Delta\varphi\|^2 + \frac12\|\nabla \sigma\|^2 & \leq \xi\|\sigma\|^2 +\xi^{-1}\|S\|^2+C\notag\\
 &\leq \xi C_P^2\|\nabla \sigma\|^2+\xi^{-1}\|S\|^2+C(1+\xi^{-1}).
 \label{energy e5}
 \end{align}

 Recalling Korn's inequality and Poincar\'{e}'s inequality, we get
 \be
 \|\bm{v}\| \le C_P'\|D\bm{v}\|, \label{po}
 \ee
 where the constant $C_P'$ only depends on $\Omega$. Then adding \eqref{BEL3} with \eqref{energy e5}, we can first take $\xi\in(0,1)$ to be sufficiently small and then find a positive constant $\zeta\in (0,\xi)$ such that
 \begin{align}
 &\frac{d}{dt}\widetilde{\mathcal{E}}(t)
 +\widetilde{\mathcal{D}}(t) + \zeta \widetilde{\mathcal{E}}(t)\leq  C(1+\|S\|^2),
 \label{BEL4}
 \end{align}
 where
 \begin{align}
 &\widetilde{\mathcal{E}}(t)=\mathcal{E}(t)+\frac12 \big(\| \varphi(t)\|^2+\|\sigma(t)\|^2\big)+C_5,\notag \\
 &\widetilde{\mathcal{D}}(t)=\frac12 \mathcal{D}(t)+ \frac{B}{4}\|\Delta \varphi(t)\|^2+ \frac{1}{4}\|\nabla \sigma(t)\|^2.\notag
 \end{align}
 The constants $\xi$, $\zeta$ as well as $C$ only depend on the coefficients of the system, $\Omega$, $\overline{\sigma_0}$, $S_0$. $C_5$ is a positive constant that guarantees $\widetilde{\mathcal{E}}(t)\geq 1$ and depends on $\varPsi$, $A$, $\chi$, $\Omega$, $\overline{\sigma_0}$, $S_0$, but not on other information of the initial data.

Applying a suitable Gronwall's lemma (see, e.g., \cite[Lemma 2.5]{GGP}) to the differential inequality \eqref{BEL4}, we deduce from the assumption \eqref{ass1} that for all $t \geq 0$:
 \begin{equation}
 \begin{aligned}
 \widetilde{\mathcal{E}}(t)
 \leq \widetilde{\mathcal{E}}(0) \mathrm{e}^{-\zeta t} + C \int_0^te^{-\zeta(t-s)}(1+\|S(s)\|^2)\,ds \leq \widetilde{\mathcal{E}}(0) \mathrm{e}^{-\zeta t}+C_6,
 \label{bd2}
 \end{aligned}
 \end{equation}
 and
 \begin{align}
 \int_{t}^{t+1} \widetilde{\mathcal{D}}(s) \, d s \leq \widetilde{\mathcal{E}}(0) \mathrm{e}^{-\zeta t}+C_7,
 \label{bd2a}
 \end{align}
 where the constants $C_6$, $C_7$ depend on the coefficients of the system, $\Omega$,  $\overline{\sigma_0}$, $S_0$, but not on other information of the initial data.
  Hence, in view of the definitions of $\widetilde{\mathcal{E}}(t)$ and $\widetilde{\mathcal{D}}(t)$, we can conclude from the dissipative estimates \eqref{bd2}, \eqref{bd2a} that
\begin{align}
 	&\|\bm{v}(t)\|^2 +\|\varphi(t)\|_{H^1}^2+\|\sigma(t)\|^2 \notag \\
 &\quad +\int_{t}^{t+1}\big(\|\nabla\bm{v}(s)\|^2 +\|\Delta\varphi(s)\|^2 + \|\nabla \mu(s)\|^2+ \|\nabla \sigma(s)\|^2\big)\,  d s \notag \\
 	& \le C,\quad \forall\, t\geq 0,
 \notag
 \end{align}
 where the constant $C$ is independent of $t$.

 Using the above uniform-in-time estimate and by the same argument as in \cite[Section 3.2]{H} (see also \cite{GGM2017,GGW} for the cases without coupling with $\sigma$), it is straightforward to check that for all $t \geq 0$,
 \begin{align}
 &\int_{t}^{t+1}\big(\|\varphi(s)\|_{H^2}^4+\| \mu(s)\|^{2}_{H^1(\Omega)}+ \|\varphi(s)\|^2_{W^{2,q}}+\|\varPsi_0'(\varphi(s))\|^2_{L^q}\big) \, d s \leq C,  \notag\\
 &\int_{t}^{t+1}\big( \left\|\partial_{t} \varphi(s)\right\|_{(H^1)'}^{2} +\left\| \partial_{t} \sigma(s)\right\|_{(H^1)'}^2\big)\,ds \leq C,\notag
 \end{align}
 where $q\in [2,+\infty)$ and the positive constant $C$ is independent of $t$.
  As a consequence, we arrive at our conclusions \eqref{uni1} and \eqref{uni2}.
 \end{proof}

\textbf{Proof of Theorem \ref{phase}, Part II}.
Using the uniform-in-time estimates \eqref{uni1}, \eqref{uni2} and following the procedure that we derive the differential inequality \eqref{2menergy}, we see that \eqref{2menergy} still holds with the constant $C$ now  being independent of $t$. Thus, the standard Gronwall lemma yields a local estimate for $\Lambda_1(t)$ on $[0,1]$, while the uniform Gronwall lemma (see \cite[Chapter III, Lemma 1.1]{T}) gives $\Lambda_1(t)\leq C$ for all $t\geq 1$. Combining these two estimates we see that $\Lambda_1(t)$ is uniformly bounded on the whole interval $[0,+\infty)$. Thus, together with \eqref{uni1}, \eqref{uni2}, we deduce that
\be
\|\mu(t)\|_{H^1}+\|\sigma(t)\|_{H^1}\leq C, \quad \forall\, t\geq 0,
\ee
where the positive constant $C$ is independent of $t$. Thanks to Lemma \ref{se}, we then find a constant $\widehat{\delta}\in (0,1)$ such that the uniform separation property holds:
\be
\|\varphi(t)\|_{L^{\infty}} \leq 1-\widehat{\delta},\quad \forall\, t \in[0,+\infty).\notag
\ee
The proof of Theorem \ref{phase} is complete.
\hfill$\square$

 \section{Continuous Dependence}\label{cd}
 \setcounter{equation}{0}
In this section we derive a continuous dependence estimate with respect to the initial data and the source term for global strong solutions in natural energy norms. A weaker continuous dependence result has been obtained in \cite[Theorem 2.2]{H} for global weak solutions in the norm of the dual space $\bm{H}^1_{0,\mathrm{div}}(\Omega)'\times H^1(\Omega)'\times H^1(\Omega)'$.\medskip

\textbf{Proof of Corollary \ref{2main1}}.
Let $(\bm{v}_{1}, p_1, \varphi_{1},\mu_1,\sigma_{1})$ and $(\bm{v}_{2}, p_2, \varphi_{2},\mu_2,\sigma_{2})$ be two strong solutions to problem \eqref{f3.c}--\eqref{ini0} on $[0,T]$  given by Theorem \ref{2main} subject to the initial data $(\bm{v}_{01},\sigma_{01},\varphi_{01})$ and $(\bm{v}_{02},\sigma_{02},\varphi_{02})$, respectively. Denote the differences of functions by
 \begin{align*}
 &(\bm{v},p, \varphi,\mu,\sigma)
 =(\bm{v}_{1}-\bm{v}_{2},p_1-p_2, \varphi_{1}-\varphi_{2},\mu_1-\mu_2,\sigma_{1}-\sigma_{2}),\\
 &(\bm{v}_0, \varphi_0, \sigma_0,S)
 =(\bm{v}_{01}-\bm{v}_{02}, \varphi_{01}-\varphi_{02}, \sigma_{01}-\sigma_{02},S_1-S_2).
 \end{align*}
 Then it holds
 \begin{subequations}
 	\begin{alignat}{3}
 	&\partial_t  \bm{ v}+\bm{ v}_1 \cdot \nabla  \bm {v} +\bm{ v} \cdot \nabla  \bm {v}_2-\textrm{div} (  2\eta(\varphi_1) D\bm{v} )-\textrm{div} \big(  2(\eta(\varphi_1)-\eta(\varphi_2)) D\bm{v}_2 \big)+\nabla p\notag\\
 	&\quad=(\mu_1+\chi \sigma_1)\nabla \varphi+(\mu+\chi \sigma)\nabla \varphi_2\label{2atest33.c},\\
 	&\partial_t  \varphi+\bm{v}_{1} \cdot \nabla \varphi+\bm{v}\cdot \nabla \varphi_{2} =\Delta \mu -\alpha\varphi,\label{2atest11.a} \\
 &\mu=A\varPsi'(\varphi_{1})-A\varPsi'(\varphi_{2})-B\Delta \varphi-\chi \sigma, \label{2atest44.d}\\
 	&\partial_t  \sigma+\bm{v}_{1} \cdot \nabla \sigma+\bm{v} \cdot \nabla \sigma_{2}-\Delta \sigma\notag\\
 &\quad  = -\chi\Delta \varphi
 	-\mathcal{C}h(\varphi_1) \sigma + \mathcal{C}(h(\varphi_1)-h(\varphi_2))\sigma_{2}+S, \label{2atest22.b}
 \end{alignat}
 \end{subequations}
 almost everywhere in $\Omega\times (0,T)$, and
 \begin{subequations}
 	\begin{alignat}{3}
 	&\bm{v}=\mathbf{0},\quad {\partial}_{\bm{n}}\varphi={\partial}_{\bm{n}}\sigma={\partial}_{\bm{n}}\mu=0, \qquad\qquad\qquad\qquad \qquad\qquad\ \textrm{on} \   \partial\Omega\times(0,T),\\
 	&\bm{v}(0)=\bm{0},\quad \varphi(0)=0,\quad  \sigma(0)=0,\qquad\qquad\qquad \qquad\qquad\qquad \textrm{in} \   \Omega.
 	\end{alignat}
 \end{subequations}
 Multiplying \eqref{2atest33.c} by $\bm{v}$ and integrating over $\Omega$, we find
 \begin{align}
 \frac{1}{2} \frac{d}{dt}\|\boldsymbol{v}\|^{2} +2\big(\eta\left(\varphi_{1}\right) D \boldsymbol{v}, D \boldsymbol{v}\big)=J_1+J_2+J_3,
 \label{vt1}
 \end{align}
 where
 \begin{equation*}
 \begin{aligned}
 J_{1} &=-\left(\boldsymbol{v}_1 \cdot \nabla \boldsymbol{v}, \boldsymbol{v}\right) -\left(\boldsymbol{v} \cdot \nabla \boldsymbol{v}_{2}, \boldsymbol{v}\right)=-\left(\boldsymbol{v} \cdot \nabla \boldsymbol{v}_{2}, \boldsymbol{v}\right), \\ J_{2}&=-2\left(\left(\eta\left(\varphi_{1}\right)-\eta\left(\varphi_{2}\right)\right) D \boldsymbol{v}_{2}, \nabla \boldsymbol{v}\right), \\
 J_3&=((\mu_1+\chi \sigma_1)\nabla \varphi,\bm{v})+((\mu+\chi \sigma)\nabla \varphi_2,\bm{v}).
 \end{aligned}
 \end{equation*}
 In the term $J_1$ we have used the fact $\left(\boldsymbol{v}_1 \cdot \nabla \boldsymbol{v}, \boldsymbol{v}\right)=0$.
 In light of the regularity of strong solutions, we have
 \begin{align}
 J_{1}&\leq \|\boldsymbol{v}\|\left\|\nabla \boldsymbol{v}_{2}\right\|_{\bm{L}^{3} }\|\boldsymbol{v}\|_{\bm{L}^{6} } \leq \frac{\eta_{*}}{16}\|\nabla \boldsymbol{v}\|^{2}+C\left\|\nabla \boldsymbol{v}_{2}\right\|_{\bm{L}^{3}}^{2} \|\boldsymbol{v}\|^{2},\notag
 \\
 J_{2 }
 &\le \sup_{s\in[-1,1]}|2\eta'(s)|\|\varphi\|_{L^6}\|D \bm{v}_2\|_{\bm{L}^3}\|\nabla  \bm{v}\|_{\bm{L}^2}\notag\\
 & \le \frac{\eta_{*}}{16}\|\nabla \boldsymbol{v}\|^{2} +C\|D \bm{v}_2\|_{\bm{L}^3}^2\|\varphi\|_{H^1}^{2}.\notag
 \end{align}
 Then recalling the calculation in \cite[(3.24)]{H}, we can estimate $J_3$ as follows
 \begin{equation}
 \begin{aligned}
 J_{3}
 &= B \int_\Omega [(\nabla \varphi_1+\nabla \varphi_2)\otimes\nabla \varphi] :\nabla \bm{v} dx\notag\\
 &\le B (\|\nabla \varphi_1\|_{\bm{L}^4}+\|\nabla \varphi_2\|_{\bm{L}^4})\|\nabla \varphi\|_{\bm{L}^4}\|\nabla \bm{v}\|\notag\\
 &\le C\| \varphi\|_{H^2}^{\frac{1}{2}}\| \varphi\|_{H^1}^{\frac{1}{2}}\|\nabla \bm{v}\|\notag\\
 &\le \frac{\eta_{*}}{16}\|\nabla \boldsymbol{v}\|^{2} +\frac{B}{4}\|\Delta \varphi\|^{2}+C \| \varphi\|_{H^1}^2.
 \end{aligned}\end{equation}
Thanks to Korn's inequality and the above estimates, we obtain
 \begin{align}
 \frac{1}{2} &\frac{d}{dt}\|\boldsymbol{v}\|^{2} +\frac{5\eta_{*}}{8}\|\nabla\bm{v}\|^2\notag\\
 &\quad\leq \frac{B}{4}\|\Delta \varphi\|^{2} +C\left(\left\|\nabla \bm{v}_{2}\right\|_{\bm{L}^{3}}^{2}\|\bm{v}\|^{2} +\left\|\nabla  \bm{v}_{2}\right\|_{\bm{L}^{3}}^{2}\|\varphi\|_{H^1}^{2} +\|\varphi\|_{H^1}^{2}\right).
 \label{vt2}
 \end{align}

 Next, multiplying \eqref{2atest11.a} by $\varphi-\Delta \varphi$ and integrating over $\Omega$, we find
 \begin{align}
 \frac{1}{2} \frac{d}{dt}\|\varphi\|_{H^1}^{2} +(\nabla \mu, \nabla \varphi) -(\nabla \mu, \nabla\Delta \varphi) +\alpha\|\varphi\|_{H^1}^2 = J_{4}+J_5,
 \label{varphi1}
 \end{align}
 where
  \begin{align*}
 J_{4} &=\left(\varphi \boldsymbol{v}_{1}, \nabla \varphi\right) +\left(\varphi_{2} \boldsymbol{v}, \nabla \varphi\right)= \left(\varphi_{2} \boldsymbol{v}, \nabla \varphi\right),\\
 J_{5} &=-\left(\varphi \boldsymbol{v}_{1}, \nabla\Delta \varphi\right)-\left(\varphi_{2} \boldsymbol{v}, \nabla\Delta \varphi\right).
 \end{align*}
Here we have used the fact $\left(\varphi \boldsymbol{v}_{1}, \nabla \varphi\right)=0$.
We observe from equation  \eqref{2atest44.d}, Agmon's inequality in two dimensions and Young's inequality that the terms on the left-hand side of \eqref{varphi1} can be controlled as
 \begin{align}
 (\nabla \mu, \nabla \varphi)
 &=B\|\Delta\varphi\|^2+ A \big(\varPsi^{\prime\prime}(\varphi_{1})\nabla \varphi+(\varPsi^{\prime\prime}(\varphi_{1}) -\varPsi^{\prime\prime}(\varphi_{2}))\nabla \varphi_{2}, \nabla \varphi\big)\notag \\
 &\quad -\chi(\nabla \sigma, \nabla \varphi)\notag\\
 &\geq B\|\Delta\varphi\|^2-A|\theta_0|\|\nabla \varphi\|^2 -2A\max_{s\in[-1+\delta_T,1-\delta_T]}|\varPsi'''(s)|\|\nabla \varphi_2\|\|\varphi\|_{L^\infty}\|\nabla \varphi\|
 \notag\\
 &\quad -|\chi|\|\Delta \varphi\|\|\sigma\|\notag\\
 &\geq \frac{3B}{4}\|\Delta \varphi\|^{2}-C\big( \|\varphi\|_{H^1}^2+\|\sigma\|^2\big),
 \notag
 \end{align}
and in a similar manner,
 \begin{align}
 -(\nabla\mu,\nabla\Delta \varphi)&=B\|\nabla\Delta \varphi\|^2 - A\big(\varPsi^{\prime\prime}(\varphi_{1})\nabla \varphi+((\varPsi^{\prime\prime}(\varphi_{1})-\varPsi^{\prime\prime}(\varphi_{2})) \nabla \varphi_{2},\nabla\Delta \varphi)\big)\notag\\
 &\quad+\chi( \nabla\sigma,\nabla\Delta \varphi )\notag\\
 &\ge B\|\nabla\Delta \varphi\|^2 -A \max_{s\in[-1+\delta,1-\delta]}|\varPsi''(s)| \|\nabla\varphi\|\|\nabla\Delta \varphi\| \notag\\ &\quad -A\max_{s\in[-1+\delta,1-\delta]} |\varPsi^{\prime\prime\prime}(s)| \|\nabla\varphi_2\| \|\varphi\|_{L^{\infty}}\|\nabla\Delta \varphi\|\notag\\
 &\quad -|\chi|\|\nabla\sigma\|\|\nabla\Delta \varphi\| \notag\\
 &\ge \frac{3B}{4}\|\nabla\Delta \varphi\|^2 -C\|\varphi\|_{H^1}^2 - C\|\nabla\sigma\|^2.\notag
 \end{align}
 On the other hand, we deduce from the fact $\|\Delta\varphi\|^2\leq \|\nabla \Delta\varphi\|\|\varphi\|$, Ladyzhenskaya, Agmon and Young's inequalities that
\begin{align}
 J_4+J_5&\le \|\bm{v}_1\|\|\varphi\|_{L^{\infty}}\|\nabla\Delta \varphi\| + \|\varphi_2\|_{L^\infty}\|\bm{v}\| \big(\|\nabla \varphi\|+\|\nabla\Delta \varphi\|\big)\notag\\
 &\leq C(\|\Delta\varphi\|+\|\varphi\|)^\frac12\|\varphi\|^\frac12 \|\nabla\Delta \varphi\| + \|\bm{v}\| \big(\|\nabla \varphi\|+\|\nabla\Delta \varphi\|\big)\notag\\
 &\le \frac{B}{4}\|\nabla\Delta \varphi\|^2+C\|\varphi\|_{H^{1}}^2+C\|\bm{v}\|^2. \notag
 \end{align}
 From the above estimates, we deduce from \eqref{varphi1} that
 \be
 \frac{1}{2} \frac{d}{dt}\|\varphi\|_{H^1}^{2} +\frac{B}{2}\|\nabla\Delta \varphi\|^2 +\frac{3B}{4}\|\Delta\varphi\|^2
 \le  C\big(\|\bm{v}\|^2+\|\varphi\|_{H^{1}}^2 +\|\sigma\|^2\big)
 + C_8\|\nabla\sigma\|^2,
 \label{varphi4}
 \ee
 where $C_8$ is a positive constant depending on norms of the initial data, coefficients of the system, $\Omega$ and  $T$.

 Adding \eqref{vt2} and \eqref{varphi4} together, we obtain
 \begin{align}
 \frac{1}{2} &\frac{d}{dt}\left(\|\boldsymbol{v}\|^{2} +\|\varphi\|_{H^1}^{2}\right) +\frac{5\eta_{*}}{8}\|\nabla\bm{v}\|^2+ \frac{B}{2}\|\nabla\Delta \varphi\|^2  +\frac{B}{2}\|\Delta \varphi\|^{2}\notag\\
 &\quad\le C\big(1+\left\|\nabla \bm{v}_{2}\right\|_{\bm{L}^{3}(\Omega)}^{2}\big) \big(\|\bm{v}\|^{2}+
  \|\varphi\|_{H^1}^{2} \big) +C\|\sigma\|^2+ C_8\|\nabla\sigma\|^2.
 \label{2vphi}
 \end{align}
 Now multiplying \eqref{2atest22.b} by $\sigma$ and integrating over $\Omega$, we find
 \begin{align}
 \frac{1}{2} \frac{d}{dt}\|\sigma\|^{2}+\|\nabla \sigma\|^2 =\sum_{i=6}^{9} J_{i},\label{sigma}
 \end{align}
 with
 \begin{align}
 J_6  &=-\chi(\Delta\varphi, \sigma),\notag\\
 J_{7} &=\big(\sigma \boldsymbol{v}_{1}, \nabla \sigma\big)+\big(\sigma_{2} \boldsymbol{v}, \nabla \sigma\big)=\big(\sigma_{2} \boldsymbol{v}, \nabla \sigma\big), \notag\\
 J_{8}&= -(\mathcal{C}h(\varphi_{1}) \sigma,\sigma)-(\mathcal{C}(h(\varphi_{1})- h(\varphi_{2}))\sigma_{2},\sigma),\notag\\
 J_9&=(S,\sigma),\notag
 \end{align}
 where we have used the fact $\left(\sigma \boldsymbol{v}_{1}, \nabla \sigma\right)=0$.
 The terms $J_6,...,J_9$ can be estimated by using H\"{o}lder, Ladyzhenskaya, Agmon and Young's inequalities:
 \begin{align}
 J_6&\leq \gamma_1\|\Delta \varphi\|^{2}+C\gamma_1^{-1}\|\sigma\|^2,\notag\\
 J_7 &\le\|\sigma_2\|_{L^4} \|\bm{v}\|_{\bm{L}^4}\|\nabla\sigma\| \le \frac{1}{4}\|\nabla\sigma\|^2 +\gamma_1\|\nabla\bm{v}\|^2 +C\gamma_1^{-1}\|\bm{v}\|^2,
 \notag\\
 J_8&\le |\mathcal{C}|\max_{s\in[-1,1]}|h(s)|\|\sigma\|^2+  C\max_{s\in[-1,1]}|h'(s)|\|\varphi\|_{L^\infty} \|\sigma_2\|\|\sigma\|\notag\\
 &\le \gamma_1\|\Delta \varphi\|^{2} +C(1+\gamma_1^{-1})\big(\|\sigma\|^2 +\|\varphi\|^2\big),\notag\\
 J_9&\leq \frac12\|\sigma\|^2+ \frac12\|S\|^2,\notag
 \end{align}
 where $\gamma_1$ is a positive constant to be determined below.  From the above estimates, we deduce from \eqref{sigma} that
 \begin{align}
 &\frac{1}{2} \frac{d}{dt}\|\sigma\|^{2} +\frac{3}{4}\|\nabla\sigma\|^2\notag\\
 &\quad \leq \gamma_1\|\nabla\bm{v}\|^2 +2\gamma_1\|\Delta \varphi\|^{2}+C(1+\gamma_1^{-1})(\|\bm{v}\|^2+ \|\varphi\|^2 +\|\sigma\|^2) +\frac12\|S\|^2.
 \label{sigma1}
 \end{align}
 Multiplying \eqref{sigma1} by $\gamma_2=2C_8>0$ and summing up with \eqref{2vphi}, we obtain
 \begin{align}
 & \frac{d}{dt}\mathcal{H}(t)  +  \mathcal{I}(t)  \leq  \beta(t)\mathcal{H}(t)+C_8\|S\|^2,\quad  \forall \, t\in[0,T].
 \label{2menergy2}
 \end{align}
 where
 \begin{align*}
 \mathcal{H}(t)&=\frac{1}{2}\|\boldsymbol{v}(t)\|^{2} +\frac{1}{2}\|\varphi(t)\|_{H^1}^{2} +C_8\|\sigma(t)\|^{2},
 \notag\\
 \mathcal{I}(t)& =\left(\frac{5\eta_{*}}{8 } - 2C_8\gamma_1 \right)\|\nabla\bm{v}(t)\|^2 +\frac{B}{2}\|\nabla \Delta \varphi(t)\|^{2} + \left(\frac{B}{2}-4C_8\gamma_1\right)\|\Delta \varphi(t)\|^{2} +\frac12 C_8\|\nabla\sigma(t)\|^2,\\
 \mathcal{\beta}(t)
 &=C(1+\gamma_1^{-1})\big(1+\left\|\nabla \bm{v}_{2}(t)\right\|_{\bm{L}^{3}(\Omega)}^{2}\big).
 \end{align*}
 Taking
 $$\gamma_1=\frac{1}{16 C_8}\min\{\eta_*,\,B\}>0,$$
 it follows that
 $$
 \mathcal{I}(t)\geq \frac{\eta_{*}}{2 } \|\nabla\bm{v}\|^2 +\frac{B}{2}\|\nabla \Delta \varphi\|^{2} + \frac{B}{4}\|\Delta \varphi\|^{2} +\frac12 C_8\|\nabla\sigma\|^2.
 $$
 Since $\beta(t)\in L^1(0,T)$ (recall \eqref{phih1}), we infer from Gronwall's lemma that
 \begin{align}
 &\|\boldsymbol{v}(t)\|^{2}+\|\varphi(t)\|_{H^1}^{2} +\|\sigma(t)\|^{2} +\int_{0}^t\Big(\|\boldsymbol{v}(s)\|^{2}_{\bm{H}^1} +\|\varphi(s)\|_{H^3}^{2} +\|\sigma(s)\|^{2}_{H^1}\Big)\, d s\notag\\
 &\quad\le C_T\Big( \|\boldsymbol{v}_0\|^{2} +\|\varphi_0\|_{H^1}^{2}+\|\sigma_0\|^{2}+\int_0^t \|S(s)\|^2\,ds\Big),\quad \forall\,t\in[0,T],\label{bd1}
 \end{align}
 where $C_T$ is a positive constant depending on norms of the initial data, coefficients of the system, $\Omega$ and $T$. The proof of Corollary \ref{2main1} is complete.
 \hfill $\square$

\section{Propagation of Regularity for Weak Solutions}
\label{phsw}
\setcounter{equation}{0}

In this section, we prove the regularity of global weak solutions to problem \eqref{f3.c}--\eqref{ini0} for positive time. \smallskip

\textbf{Proof of Corollary \ref{regw}}. Let $(\bm{v}, \varphi, \mu, \sigma)$ be a global weak solution to problem \eqref{f3.c}--\eqref{ini0} on $[0,T]$. Thanks to its  regularity presented in Definition \ref{maind}, we are allowed to multiply \eqref{f3.c} by $\bm{v}$, \eqref{f1.a} by $\mu$, \eqref{f4.d} by  $\partial_t \varphi$, \eqref{f2.b} by $\sigma+\chi(1-\varphi)$, and integrate over $\Omega$ to obtain the energy identity \eqref{BEL}, namely,
 \begin{align}
 \frac{d}{dt}\mathcal{E}(t)
 +\mathcal{D}(t)+\int_\Omega \alpha(\varphi-c_0)\mu \, dx  =\int_\Omega (-\mathcal{C} h(\varphi) \sigma +S)(\sigma +\chi(1-\varphi)) dx,
 \label{BEL5}
 \end{align}
 for a.e. $t\in (0,T)$, where $\mathcal{E}(t)$ and $\mathcal{D}(t)$ are given by \eqref{E}, \eqref{D}, respectively.

 Similar to \cite[Theorem 4.2]{GMT}, we can deduce from \eqref{mph2} and \eqref{BEL5} that for any $\tau\in (0,T)$, there exists $\tau_1\in (0, \tau)$ such that (see also \eqref{estimate1})
 $\bm{v}(\tau_1)\in \bm{H}^1_{0,\mathrm{div}}(\Omega)$,
 $\varphi(\tau_1) \in H^{2}_N(\Omega)$,
 $\sigma(\tau_1) \in H^1(\Omega)$
 with $\left\|\varphi(\tau_1)\right\|_{L^{\infty}} \leq 1,\ |\overline{\varphi}(\tau_1)|<1$ and
	$\mu(\tau_1)=A\varPsi^{\prime}\left(\varphi(\tau)\right) -B\Delta \varphi(\tau_1) -\chi\sigma(\tau_1)\in H^1(\Omega)$, and the corresponding norms may depend on $\|\bm{v}_0\|$, $\|\varphi_0\|_{H^1}$, $\|\sigma_0\|$, $\int_\Omega \varPsi(\varphi_0)dx$, coefficients of the system, $\Omega$ and the time $\tau_1$.
Then taking $(\bm{v}(\tau_1), \varphi(\tau_1), \sigma(\tau_1))$ as an initial datum, we apply Theorem \ref{2main} to obtain a unique strong solution defined on $[\tau_1,T]$, which indeed coincides with the previous weak solution due to the uniqueness result \cite[Theorem 2.2]{H}. Furthermore, by Theorem \ref{phase}, we see that there exists  $\widetilde{\delta}\in (0,1)$ such that $\|\varphi(t)\|_{L^{\infty}} \leq 1-\widetilde{\delta}$ for $[\tau_1,T]$, where $\widetilde{\delta}$ may depend on strong norms of $(\bm{v}(\tau_1), \varphi(\tau_1), \sigma(\tau_1))$ (and thus on $\tau_1$), coefficients of the system, $\Omega$ and $T$. Since $\tau\in (0,T)$ is arbitrary, we have thus shown that the weak solution to problem \eqref{f3.c}--\eqref{ini0} regularizes instantaneously for positive time to be a strong solution and the strict separation from pure phases $\pm 1$ takes place as long as $0<t\leq T$. Thus, the proof of Corollary \ref{regw} is complete. \hfill $\square$

\begin{remark}
Under suitable assumptions like \eqref{ass1}, in view of the dissipative estimates \eqref{bd2}, \eqref{bd2a} and the regularizing effects illustrated above, we can further analyze the long-time behavior of problem \eqref{f3.c}--\eqref{ini0}, for instance, the attractor theory within the framework of infinite dimensional dynamical systems as well as the convergence to a single equilibrium for every bounded trajectories. This will be studied in a forthcoming paper.
\end{remark}

\section*{Acknowledgments}
\noindent
The authors would like to thank the anonymous referee for several helpful comments.
The second author is partially supported by NNSFC grant No. 11631011 and the Shanghai Center for Mathematical Sciences at Fudan University.


%
\end{document}